\documentclass[12 pt, english]{article}
\usepackage[T1]{fontenc}
\usepackage{ulem}
\usepackage[utf8]{luainputenc}
\usepackage{geometry}
\usepackage[pdftex]{graphicx}
\usepackage{amstext}
\usepackage{amsthm}
\usepackage{amssymb}
\usepackage{amsmath}
\usepackage[T1,T2A]{fontenc}
\usepackage[utf8]{inputenc}
\usepackage[english]{babel}
\usepackage{esint}
\usepackage{babel}
\usepackage{float}
\usepackage{amsfonts}
\usepackage{fullpage} 
\usepackage{indentfirst}
\usepackage[colorlinks]{hyperref}
\usepackage{cmap}

\renewcommand{\Re}{\mathop{\mathrm{Re}}\nolimits}

\newtheorem{remark}{Remark}
\newtheorem{lemma}{Lemma}
\theoremstyle{definition}

\newtheorem{theorem}{Theorem}
\newtheorem{definition}{Definition}

\numberwithin{equation}{section}
\numberwithin{theorem}{section}
\numberwithin{remark}{section}
\numberwithin{proposition}{section}
\numberwithin{lemma}{section}
\numberwithin{corollary}{section}
\numberwithin{definition}{section}

\title{Large deviations for number of irreducible divisors\\of the Dirichlet series distribution}

\author{
  Michael~Cranston\footnote{University of California, Irvine, United States of America.
  E-mail: mcransto@uci.edu.}
  \and
  Mariia~Khodiakova\footnote{Universit\"{a}t Z\"{u}rich, Schweiz. 
  E-mail: mariia.khodiakova@math.uzh.ch.}}

\date{}

\begin{document}

\maketitle

\begin{abstract}
In this paper we produce precise large deviation estimates through the lens of mod-Poisson convergence. We apply a general result to various examples from number theory, Dedekind domains and polynomials over finite fields when an element is selected using a distribution based on a Dirichlet series.
\end{abstract}



\section{Introduction}

In this paper we use Dirichlet series to sample elements from  commutative semigroups  $R$ that are endowed with an addition operation $+$ and multiplication operator $\cdot$, which is distributive over $+$, with a multiplicative identity element $1$, and without zero divisors. We assume a statement analogous to the Fundamental Theorem of Arithmetic holds, i.e. there is some set of prime elements and unique factorization, holds in $R$. In addition, we will need some notion of size of the elements that is a positive  integer valued, multiplicative function on $R.$  This role will be played by  the function $b$ which appears in definition of the Dirichlet series at (\ref{defA}) below.
For instance, $R$ could be the set of natural numbers $\mathbb{N}$ or a unique factorization domain (factorial ring) ($\mathbb{Z}$, $\mathbb{Z}[x]$) or a Dedekind domain. We shall further assume that the set of multiplicative invertible elements $E$ is finite;
$R$ can have an additive identity element $0$, but this is not required. 
That is why, further, if there is no $0$ in $R$ we will mean $R$ writing $R \setminus \{0\}$.

On account of our assumption about the Fundamental Theorem of Arithmetic, there is a well defined  greatest common divisor of two elements, which is the set of common divisors that are divisible by all other common divisors.
Using that we can introduce a notion analogous to that of a multiplicative function on the  positive integers.
\begin{definition}
A non-negative real-valued function $a$ is multiplicative on $R \setminus \{0\}$ if $a(n m) = a(n) a(m)$ for any $n, m \in R \setminus \{0\}$ such that $\gcd (n, m) = E$ 
and $a(e) = 1$ for any $e \in E$.  

The function $a(n)$ is completely multiplicative on $R \setminus \{0\}$ if $a(n m) = a(n) a(m)$ for any $n, m \in R \setminus \{0\}$ and $a(e) = 1$ for any $e \in E$.  
\end{definition}

Suppose we are given a commutative semigroup $R$ as above and let $a$ and $b$ be multiplicative functions on $R \setminus \{0\}$ satisfying  $a(n) \geq 0$ and $b(n) \ge 1$ for all $n\in R\setminus\{0\}$ .
Then define a Dirichlet series by
\begin{eqnarray}\label{defA}
A(z) = \sum_{n \in R \setminus \{0\}} \frac{a(n)}{b(n)^z},
\end{eqnarray}
that is an analytic function in $\Re z > \kappa$ for some real $\kappa$. 

In the article \cite{CM1}, the authors considered the Riemann zeta distribution of a positive integer-valued random variable. 
Similar to that work, we will consider the distribution of the $R$-valued non-zero random variable $X_s$, $s > \kappa$, based on the Dirichlet series and defined by
\begin{gather} \label{distr}
\mathbf{P}(X_s = n) = \frac{a(n)}{A(s) \, b(n)^s}, \quad n \in R \setminus \{0\}.
\end{gather}
We will call this the Dirichlet series distribution. In our applications, the quantity $b(n)$ can be thought of as a measure of the size of the element $n.$

Let $\mathcal{P}$ be the set of irreducible elements in $R$. 
Now we obtain some basic properties of this distribution. 
For any irreducible element $p \in \mathcal{P}$ and natural number $\alpha$ we get the following:
\begin{eqnarray*}
\begin{split}
\mathbf{P}\left(p^{\alpha} | X_s\right) =& \sum_{k \in R \setminus \{0\}} \mathbf{P}\left(X_s = p^{\alpha} k\right) \\
= &\frac{1}{A(s)} \sum_{k \in R \setminus \{0\}} \frac{a\left(p^{\alpha} k\right)}{b(p^{\alpha} k)^s} \\
= &\frac{1}{A(s)} \sum_{k : \gcd(k,p) = E} \sum_{m=0}^\infty\frac{a\left(p^{\alpha+m} k\right)}{b(p^{\alpha+m}k)^s} \\
= &\frac{1}{A(s)} \sum_{k : \gcd(k,p) = E} \sum_{m=0}^\infty\frac{a\left(p^{\alpha+m} \right)a\left(k \right)}{b(p^{\alpha+m} )^sb(k)^s} \\
= &\frac{1}{A(s)}  \sum_{m=0}^\infty \frac{a\left(p^{\alpha+m} \right)}{b(p^{\alpha+m} )^s} \sum_{k : \gcd(k,p) = E} \frac{a\left(k \right)}{b(k)^s} \\
= &\frac{A_{p^{\alpha}}(s)}{1 + A_p(s)},
\end{split}
\end{eqnarray*}
where
\begin{gather} \label{func_Ap}
A_{p^{\alpha}}(s) = \sum_{m=0}^{\infty} \frac{a\left(p^{\alpha+m}\right)}{b(p^{\alpha + m})^s} = 
\sum_{m=\alpha}^{\infty} \frac{a\left(p^{m}\right)}{b(p^m)^s},
\quad \alpha \in \mathbb{N}.
\end{gather}
A similar computation shows that for any distinct irreducible elements $p, q \in \mathcal{P}$ and natural numbers $\alpha$, $\beta,$
\begin{gather*}
\mathbf{P}\left(p^{\alpha} | X_s, q^{\beta} | X_s\right) = \mathbf{P}\left(p^{\alpha} | X_s\right) \mathbf{P}\left(q^{\beta} | X_s\right).
\end{gather*}

We consider two classic functions of number theory. 
The first one is the number of distinct irreducible divisors
\begin{eqnarray*}\label{omega}
\omega(n) = \# \{p \in \mathcal{P} : p | n\} = 
\sum_{p \in \mathcal{P}} I_{p | n},
\end{eqnarray*}
and the second one is the number of irreducible divisors counted with multiplicities
\begin{eqnarray*}\label{Omega}
\Omega(n) = \# \{p \in \mathcal{P}, \alpha \in \mathbb{N} : p^{\alpha} | n\} = 
\sum_{p \in \mathcal{P}} \sum_{\alpha=1}^{\infty} I_{p^{\alpha} | n}.
\end{eqnarray*}
In the papers \cite{CP} , \cite{CM1} and \cite{CM2},  a central limit theorem was proven for $\omega(X_s)$ and $\Omega(X_s).$ When $R$ is the set of positive integers, $a(n)\equiv 1$ and $b(n)=n$ and $s\searrow 1$ the function $A$ at (\ref{defA}) is the Riemann zeta function. In contrast, we also consider $\omega(X_s)$, $\Omega(X_s)$ as $s \searrow \kappa$, but compare their distribution with the Poisson distribution instead of the normal distribution and formulate some limit theorems using  large deviation theory obtained using mod-Poisson convergence.

The work is organized as follows. 
In section \ref{sec2} we consider the necessary information that will be used in the paper.
In section \ref{sec3} we formulate and prove the main results about the number of irreducible divisors of $X_s$, counted without multiplicities.
In section \ref{sec4} we obtain similar results for the number of irreducible divisors of $X_s$, counted with multiplicities.
In section \ref{sec5} we consider examples of various integer-valued Dirichlet series distributions. 
In section \ref{sec6} we consider non-numeric applications for ideals in Dedekind domains and random polynomials with coefficients in a finite field.

\section{Preliminaries} \label{sec2}

Let $\{\xi_n: n\ge 1\}$ be a sequence of real-valued random variables, and denote by
$\phi_n (z) = \mathbf{E} e^{z \xi_n}$ their moment generating functions. We suppose these exist for all $n$ in a fixed strip
\begin{gather*}
\mathcal S_{(c,d)} = \{z: c < \Re z < d\}, \quad c < 0 < d.
\end{gather*}
We also assume that $\phi$ is a non-constant infinitely divisible distribution  with moment generating function
\begin{gather*}
e^{\eta(z)} \equiv \int_{\mathbb{R}} e^{z x} \phi(dx)
\end{gather*} 
that is well defined on $\mathcal S_{(c,d)}$, and that $\psi(z)$ is an analytic function  that does not vanish on the real part of $\mathcal S_{(c,d)}.$ 

\begin{definition}
We say that the random sequence $\{\xi_n: n\ge 1\}$ converges mod-$\phi$ as $n \to \infty$ on $\mathcal S_{(c,d)}$ 
with parameters $t_n$ and limiting residue function $\psi$ if
\begin{gather*} \label{sec2:eq1}
\psi_n(z) \equiv e^{- t_n \eta(z)} \phi_n(z) \to \psi(z), 
\quad n \to \infty,
\end{gather*}
locally uniformly over $z \in \mathcal S_{(c,d)}$,
where $t_n$ is some sequence tending to $+\infty$. 

In addition, we say that the sequence of random variables 
$\{\xi_n: n\ge 1\}$ converges mod-$\phi$ at speed 
$O\left(t_n^{-v}\right)$ if for any compact subset $K$ of $\mathcal S_{(c,d)}$ there is a positive constant $C_K$ such that
\begin{gather*}
\left|\psi_n(z) - \psi(z)\right| < C_K t_n^{-v}
\end{gather*}
for any $z \in K.$
\end{definition}
We can use the similar definition with the $o(\cdot)$
notation.  In this paper we will be dealing not with a sequence of random variables, but rather with a family of random variables parametrized by real values $s$ greater than some value $\kappa.$ Then the limit in the definition of mod-Poisson convergence will be as $s$ goes down to $\kappa$ rather than as $n$ goes to $\infty,$
 all else being the same.
In the general case, this notation was used in the papers 
\cite{FMN1} and \cite{FMN2}.
When $\phi$ is the standard Gaussian, 
in that case $\eta(z) = z^2/2,$ or standard Poisson distribution, where
$\eta(z) = e^z - 1,$  we talk about mod-Gaussian or mod-Poisson convergence, respectively. 
One can read more about these convergences in the articles 
\cite{JKN} or \cite{KN1}, respectively.

The notion of mod-$\phi$ convergence is related to the 
Legendre-Fenchel transform. The Legendre-Fenchel transform of a function $\eta$ is defined by
\begin{gather}\label{L-F}
F(x) = \sup_{h \in \mathbb{R}} (h x - \eta(h)).
\end{gather}
This is an involution on convex lower semi-continuous functions.
For instance, if we consider mod-Poisson convergence then $\eta(h)=e^h-1$ and the $\sup$ is attained at $h = \ln x,$ so that 
\begin{gather*}
 \quad F(x) = 
\begin{cases}
x \ln x - x+1, \quad x > 0, \\
+\infty, \quad x \leq 0.
\end{cases}
\end{gather*}

Also, there is a connection between mod-$\phi$ convergence and the Cramer transform. 
The Cramer transform of a random variable $\xi$ with a parameter 
$h \in [H^-, H^+]$ for some $H^- \leq 0 \leq H^+$ is defined by
\begin{gather*}
\mathbf{P} \left(\xi^{(h)}\in dx\right) = R(h)^{-1} e^{hx} \mathbf{P}(\xi \in dx),
\quad R(h) = \int_{\mathbb{R}} e^{hx} \mathbf{P}(\xi \in dx) = \mathbf{E} e^{h \xi} < \infty.
\end{gather*}
The distribution of the random variable
$\xi^{(h)}$ is called the conjugate distribution with a parameter $h$ with respect to the distribution of $\xi$. 
We can obtain that for any parameter $h \in [H^-, H^+]$ the following is true: 
\begin{gather*}
\mathbf{E} \xi^{(h)} = \left(\ln{R(h)}\right)', \quad
\mathbf{D} \xi^{(h)} = \left(\ln{R(h)}\right)''. 
\end{gather*}
Note that if $\psi \equiv 1$ then the function $\ln R(h)$ corresponds to $\eta(h)$ and we can consider Legendre-Fenchel transform of the function $\ln R(h)$.

We will need the following statements for lattice random variables from the paper \cite{FMN2}.

\begin{lemma} \label{sec2:lem1}
Let $\xi$ be a $\mathbb Z$-valued random variable whose generating function $\phi_{\xi} (z) = \mathbf{E} e^{z \xi}$
converges absolutely in the strip $\mathcal S_{(c,d)}$. 
Fix $k \in \mathbb{Z}.$  Then
\begin{eqnarray*}\label{fourier1}
\forall h \in (c, d) \quad 
\mathbf{P} (\xi = k) = \frac{1}{2 \pi} \int_{-\pi}^{\pi}
e^{- k (h + i u)} \phi_{\xi} (h + i u) du, 
\end{eqnarray*}
\begin{eqnarray*}\label{fourier2}
\forall h \in (0, d) \quad
\mathbf{P} (\xi \geq k) = \frac{1}{2 \pi}
\int_{-\pi}^{\pi}
\frac{e^{- k (h + i u)}}{1 - e^{-(h+iu)}} \phi_{\xi} (h + i u) du.
\end{eqnarray*}
\end{lemma}

\begin{theorem} \label{sec2:th1}
Suppose a sequence of random  variables $\{\xi_n: n\ge 1\}$ converges mod-$\phi$ with parameters $t_n$ and residue function $\psi$ at speed $o(1)$.
Let $F$ be given by (\ref{L-F}) and $h$ be defined by the implicit equation 
$\eta'(h) = x$. 
\begin{enumerate}
\item The following expansion holds:
\begin{eqnarray*}\label{ld1}
\mathbf{P} (\xi_n = t_n x) = \frac{e^{-t_n F(x)}}{\sqrt{2 \pi t_n \eta''(h)}} \psi(h) (1 + o(1)), \quad n \to \infty,
\end{eqnarray*}
where $o(1)$ is uniformly small over $x \in (\eta'(c), \eta'(d))$,
$t_n x \in \mathbb N$.
\item Similarly, 
\begin{eqnarray*}\label{ld2}
\mathbf{P} (\xi_n \geq t_n x) = \frac{e^{-t_n F(x)}}{\sqrt{2 \pi t_n \eta''(h)}} \frac{1}{1-e^{-h}}\psi(h) (1 + o(1)), \quad n \to \infty,
\end{eqnarray*}
where $o(1)$ is uniformly small over $x \in (\eta'(0), \eta'(d))$,
$t_n x \in \mathbb N$.
\item By applying the result to $-\xi_n$, one gets 
\begin{eqnarray*}\label{ld3}
\mathbf{P} (\xi_n \leq t_n x) = \frac{e^{-t_n F(x)}}{\sqrt{2 \pi t_n \eta''(h)}} \frac{1}{1-e^{-|h|}} \psi(h) (1 + o(1)), \quad n \to \infty,
\end{eqnarray*}
where $o(1)$ is uniformly small over $x \in (\eta'(c), \eta'(0))$,
$t_n x \in \mathbb N$.
\end{enumerate}
\end{theorem}

\section{Number of distinct irreducible divisors of $X_s$} \label{sec3}

Now we consider the random variables $\omega(X_s)$ where $X_s$,
$s > \kappa$, has the Dirichlet series distribution at (\ref{distr}). 
The expectation and variance of $\omega(X_s)$ are given by the following:
\begin{gather*}
\mathbf{E} \omega(X_s) = \mathbf{E} \sum_{p \in \mathcal{P}} I_{p | X_s} = 
\sum_{p \in \mathcal{P}} \frac{A_{p}(s)}{1+A_{p}(s)} =: \mathbb{P} (s), \\
\nonumber
\mathbf{D} \omega(X_s) = \mathbf{E} \omega^2 (X_s) - \left(\mathbf{E} \omega(X_s)\right)^2 = 
\mathbb{P}(s) - \sum_{p \in \mathcal{P}} \left(\frac{A_{p}(s)}{1+A_{p}(s)}\right)^2.
\end{gather*}
The quantity $A_p(s)/(1+A_p(s))$ comes up frequently so we shall simplify notation by defining
\begin{eqnarray*}\label{defalp}
\alpha_p(s)=\frac{A_p(s)}{1+A_p(s)}.
\end{eqnarray*}

We will make the following assumptions:
\begin{enumerate}
\item[${\bf{A1}}$] There exists a positive number $\kappa$ such that $A(s)$ converges for any $s > \kappa$.

\item[${\bf{A2}}$] $A(s) \nearrow +\infty$ as $s \searrow \kappa$.

\item[${\bf{A3}}$] There exists an enumeration of irreducible elements $\mathcal{P} = \{p_i, i \in \mathbb{N}\}$ such that 
\begin{eqnarray} 
\nonumber
b(p_i) \to \infty, \quad i \to \infty, \\
\label{APbdd}
A_{p_i}(\kappa) \to 0, \quad i \to \infty.
\end{eqnarray}
From now on, our enumerations of $\mathcal{P}$ we will satisfy this condition.

\item[${\bf{A4}}$] For any natural numbers $n \geq 2$, $l \geq 1$ and real $s \geq \kappa,$ 
\begin{eqnarray}\label{A4}
\sum_{i > l} \alpha^n_{p_i}(s) \leq \frac{C(n,l)}{b(p_l)^{\,n (s - (\kappa-1))}}
\end{eqnarray}
for some positive function $C(n, l)$ satisfying
\begin{eqnarray}\label{A4a}
\lim_{l\to\infty}\sup_{n\ge1}\frac{C(n,l)}{b(p_l)^{\,n (s - (\kappa-1))}} = 0, 
\end{eqnarray}
and  for all $l$ and for some $c > 0,$
\begin{eqnarray}\label{A4b}
\limsup_{n \to \infty} C(n,l)^{1/n} \leq 1/c.
\end{eqnarray}

\item[${\bf{A5}}$] For any $p \in \mathcal{P},$ the derivative $A'_p(s)$ exists for any $s \geq \kappa,$ and
\begin{eqnarray*}\label{A5}
\sum_{p \in \mathcal{P}} A'_p(s) A_p(s)<\infty.
\end{eqnarray*}
\end{enumerate}

A quantity that plays a large role in this paper is
\begin{gather} \label{func_Pn}
\mathbb{P}_n(s) \equiv \sum_{p \in \mathcal{P}} \alpha^n_p(s), \quad n \geq 1,
\end{gather}
where
\begin{eqnarray*}
 \mathbb{P}_1(s) \equiv \mathbb{P}(s).
\end{eqnarray*}

Observe that due to ${\bf{A4}}$ we get 
\begin{eqnarray*}\label{varest}
\mathbf{D} \omega(X_s) = \mathbb{P}(s) - \mathbb{P}_2(s) = \mathbb{P}(s) + O(1), \quad s \searrow \kappa.
\end{eqnarray*}

The next Lemma is the analog of Euler's famous  equation:
\begin{eqnarray*}
\ln \zeta(s)=\sum_p\frac{1}{p^s}+ O(1), \quad s \searrow 1,
\end{eqnarray*}
where $\zeta$ is the Riemann zeta function.
\begin{lemma} \label{sec3:lem1}
Assume conditions ${\bf{A1-A4}}$ hold. Then
$\mathbb{P}(s)$ converges for any $s > \kappa$ and 
\begin{gather*}
\mathbb{P}(s) \asymp \ln A(s), \quad s \searrow \kappa.
\end{gather*}
\end{lemma}
\begin{proof}
Since $a$ and $b$ are multiplicative functions,
\begin{eqnarray*}
A(s) =& \prod_{p \in \mathcal{P}} \left(1 + \sum_{\alpha=1}^{\infty} \frac{a(p^{\alpha})}{b(p^{\alpha})^s}\right)\\ 
=&\prod_{p \in \mathcal{P}} \left(1 + A_p(s)\right)\\ 
=& \exp \left(\sum_{p \in \mathcal{P}} \ln \left(1 + A_p(s)\right)\right).
\end{eqnarray*}

Due to condition ${\bf{A3}}$, (\ref{APbdd}), $\exists \, N \in \mathbb{N}$ such that  $A_{p_i}(s) \leq A_{p_i}(\kappa) < 1$ for all  $i > N.$

Then we obtain the following:
\begin{eqnarray}\label{As}
\begin{split}
A(s) = &\exp \left(\sum_{i \leq N} \ln \left(1 + A_{p_i}(s)\right)\right) \exp \left(\sum_{i > N} \ln \left(1 + A_{p_i}(s)\right)\right) \\ 
=& \exp \left(\sum_{i \leq N} \ln \left(1 + A_{p_i}(s)\right)\right) \exp \left(\sum_{i > N} \left(A_{p_i}(s) - \frac{A^2_{p_i}(s)}{2} + o\left(A^2_{p_i}(s)\right)\right)\right) \\
= &\exp \left(\sum_{i \leq N} \left(\ln \left(1 + A_{p_i}(s)\right) - A_{p_i}(s)\right) - \sum_{i > N} \frac{A^2_{p_i}(s)}{2} + o(1)\right) \exp \left(\sum_{i=1}^{\infty} A_{p_i}(s)\right),
\end{split}
\end{eqnarray}
where according to ${\bf{A4}}$ the first factor of the last equality in (\ref{As}) is bounded as $s\searrow \kappa.$

This implies that 
$\ln A(s) =   \sum_{i=1}^{\infty} A_{p_i}(s)+O(1).$ 
But
\begin{eqnarray*}
 \sum_{i=1}^{\infty} A_{p_i}(s)= &\sum_{i=1}^{\infty} \alpha_{p_i}(s)(1+A_{p_i}(s))\\
 =&\sum_{i=1}^{\infty} \alpha_{p_i}(s)+\sum_{i=1}^{\infty}\frac{A^2_{p_i}(s))}{1+A_{p_i}(s)}
\end{eqnarray*}
and by (\ref{APbdd}) and (\ref{A4})
\begin{eqnarray*}
\sum_{i=1}^{\infty}\frac{A^2_{p_i}(s))}{1+A_{p_i}(s)}\le &\sup_p (1+A_{p}(\kappa))\sum_{i=1}^{\infty} \alpha^2_{p_i}(s)
=O(1).
\end{eqnarray*}
Thus, $\mathbb{P}(s)$ converges for any $s > \kappa$ and
\begin{gather*}
\ln A(s) \asymp  \left(\sum_{p \in \mathcal{P}} A_p(s)\right)
\asymp \mathbb{P}(s), 
\quad s \searrow \kappa.
\end{gather*}

Lemma \ref{sec3:lem1} is proved.
\end{proof}

We now consider the mod-$\phi$ convergence for $\omega(X_s)$. 

\begin{theorem} \label{sec3:prop1}
Assume conditions  ${\bf{A1-A5}}$ are satisfied and let the random variables $X_s, s > \kappa$, have the Dirichlet series distribution given at (\ref{distr}). Then $\omega(X_s)$ converges mod-Poisson as $s \searrow \kappa$ on $\mathbb{C}$ with parameters $t_s = \mathbb{P}(s)$ and residue function 
\begin{eqnarray}\label{psi}
\psi(z) = \exp \left(\sum_{p \in \mathcal{P}} \left(\ln \left(1 + \alpha_p(\kappa) (e^z - 1)\right) - \alpha_p(\kappa) (e^z - 1)\right)\right).
\end{eqnarray}

The mod-Poisson convergence occurs at speed 
$o(1)$, $s \searrow \kappa$.
\end{theorem}

\begin{proof}
Due to independence of the events $\{p | X_s\}$ and $\{q | X_s\}$, where $p$, $q$ are different irreducible elements, we get the following: 
\begin{eqnarray*}
\begin{split}
\phi_s (z) =& \mathbf{E} e^{z \omega(X_s)}\\
 =& \prod_{p \in \mathcal{P}} \left(1 + \alpha_p(s) (e^z - 1)\right) \\
=& \exp \left(\mathbb{P}(s) (e^z - 1) - \mathbb{P}(s) (e^z - 1) + \sum_{p \in \mathcal{P}} \ln \left(1 + \alpha_p(s)(e^z - 1)\right)\right)\\ 
=:& e^{\mathbb{P}(s) (e^z - 1)} \psi_s (z),
\end{split}
\end{eqnarray*}
where $\eta(z) = e^z - 1$ corresponds to the standard Poisson distribution.
For mod-Poisson convergence with parameters $t_s \equiv \mathbb{P}(s)$ and residue function $\psi(z) \equiv \psi_{\kappa}(z)$ we need to prove that the function $\psi(z)$ is well defined on some strip $\mathcal{S}_{(c, d)}$ and does not vanish on the real
part of $\mathcal{S}_{(c, d)}$. 
Let $e^{\vartheta_s(z)} \equiv \psi_s (z)$, $s \geq \kappa$.
Since the point $z$ is fixed, due to condition ${\bf{A3}}$ there exists 
an $l \in \mathbb{N}$ such that 
$|1 - e^z| < c \, b(p_l)$, 
where $c$ is from condition (\ref{A4b}). 
So in this domain we get
\begin{eqnarray*}
\psi_s (z) =& \exp \left(\sum_{i \leq l}
\left(\ln \left(1 + \alpha_{p_i}(s) (e^z - 1)\right) - \alpha_{p_i}(s) (e^z - 1)\right)\right) \\
\times& \exp \left(\sum_{i > l} \left(\alpha_{p_i}(s) (1 - e^z) - \sum_{n=1}^{\infty} \alpha_{p_i}(s)^n \frac{(1-e^z)^n}{n}\right)\right) \\
= &\exp \left(\sum_{i \leq l}
\left(\ln \left(1 + \alpha_{p_i}(s) (e^z - 1)\right) - \alpha_{p_i}(s) (e^z - 1)\right)\right) \\
\times &\exp \left(-\sum_{n=2}^{\infty} \left(\sum_{i > l} \alpha_{p_i}(s)^n\right) \frac{(1-e^z)^n}{n}\right).
\end{eqnarray*}
Using condition ${\bf{A4}},$ part (\ref{A4}),
for any $l \geq 1$, $n \geq 2$ and 
$s \geq \kappa,$ we obtain
\begin{gather*}
\left|\sum_{n=2}^{\infty} \left(\sum_{i > l} \alpha^n_{p_i}(s)\right) \frac{(1-e^z)^n}{n}\right| \leq 
\sum_{n=2}^{\infty} \frac{C(n,l)}{n} \left(\frac{|1-e^z|}{b(p_l)^{s -(\kappa - 1)}}\right)^n,
\end{gather*}
where the power series on the right side of the inequality converges for $z$ in a strip containing the origin due to the root test using (\ref{A4a}) and (\ref{A4b}).
It means that $\vartheta_s(z)$ and consequently $\psi_s (z)$ are well defined and holomorphic functions in the domain 
$\mathbb{C} \setminus \mathbb{S}_s$, where 
$\mathbb{S}_s \equiv \left\{- \ln A_p(s) + i (\pi + 2 \pi k), \, p \in \mathcal{P}, \, k \in \mathbb{Z}\right\}$, $s \geq \kappa$.\\
But  since $\phi_s (z) =  \prod_{p \in \mathcal{P}} \left(1 + \alpha_p(s) (e^z - 1)\right)$ we see that $ \psi_s (z) = 0$ for any complex number $z \in \mathbb{S}_s,\,\,s \geq \kappa$. 
That is why $\psi(z)$ is well defined on $\mathbb{C}$ due to its continuity and does not vanish on the real
part of $\mathbb{C}$. 
Let us now consider the speed of mod-Poisson convergence. 
For any $z \in \mathbb{C}$ 
we have
\begin{gather*}
|\psi_s(z) - \psi(z)| = |\psi(z) (f(s, z) - 1)| =
|\psi(z) f_s'(\kappa, z)| (s-\kappa) + o(s-\kappa), \quad s \searrow \kappa,
\end{gather*}
where $f(\kappa, z) = 1$, and
\begin{gather*}
f(s, z) = \exp \left(-(\mathbb{P}(s) - \mathbb{P}(\kappa)) (e^z - 1)\right) \\
\times \exp \left(\sum_{p \in \mathcal{P}} \left(\ln \left(1 + \alpha_p(s) (e^z - 1)\right) - \ln \left(1 + \alpha_p(\kappa) (e^z - 1)\right)\right)\right).
\end{gather*}
The derivative of $f(s, z)$ with respect to  $s$ is equal to
\begin{gather*}
f_s(s, z) = f(s, z) (e^z - 1)^2 \sum_{p \in \mathcal P} \frac{A'_p(s) A_p(s)}{(1 + A_p(s))^2 (1 + A_p(s) e^z)}.
\end{gather*}
Hence, if $s = \kappa$ then
\begin{gather*}
f_s(\kappa, z) = (e^z - 1) \sum_{p \in \mathcal P} \frac{A'_p(\kappa) A_p(\kappa)}{(1 + A_p(\kappa))^2 (1 + A_p(\kappa) e^z)},
\end{gather*}
where the series is well defined for any 
$z \in \mathbb{C} \setminus \mathbb{S}_{\kappa}$ according to condition ${\bf{A5}}$ and the Weierstrass M-test.
Thus, we can define $\psi(z) f'(\kappa, z)$ at any point 
$z \in \mathbb{S}_{\kappa}$ with the limit at $z$ because any such point is a removable singularity of the holomorphic function $\psi(z) f'(\kappa, z)$.

Therefore, we obtain that
\begin{gather*}
|\psi_s(z) - \psi(z)| = O(s-\kappa) = o(1), \quad s \searrow \kappa.
\end{gather*}

Theorem \ref{sec3:prop1} is proved.
\end{proof}

We can apply Theorem \ref{sec2:th1} and get the following result about the large deviation probabilities for $\omega(X_s)$. We remark that item $3$ in the Theorem follows from item $2$ by  applying the result to $-\omega(X_s)$.
 
\begin{theorem} \label{sec3:th1}
Assume conditions ${\bf{A1-A5}}$ are satisfied and that $X_s$ has the Dirichlet series distribution given by (\ref{distr}).
Given $x,$ let $h$ be defined by the implicit equation $e^h = x$, 
and let  $\psi(\cdot)$ be the function defined at (\ref{psi}) in Theorem \ref{sec3:prop1}.
Then the following large deviation estimates hold:
\begin{enumerate}
\item For $x \in (\eta'(-\infty), \eta'(+\infty))=(0,\infty)$:
\begin{gather*}
\mathbf{P} \left(\omega(X_s) = \mathbb{P}(s) x\right) = \frac{e^{-\mathbb{P}(s) (x \ln x - x + 1)}}{\sqrt{2 \pi \mathbb{P}(s) x}} \psi(h) (1 + o(1)), \quad s \searrow \kappa,
\end{gather*}
where $o(1)$ is uniformly small over $x > 0$,
$\mathbb{P}(s) x \in \mathbb N$.
\item For $x \in (\eta'(0), \eta'(+\infty))=(1,\infty)$:
\begin{gather*}
    \mathbf{P} \left(\omega(X_s) \geq \mathbb{P}(s) x\right) = \frac{e^{-\mathbb{P}(s) (x \ln x - x + 1)}}{\sqrt{2 \pi \mathbb{P}(s) x}}
\frac{1}{1-x^{-1}} \psi(h) (1 + o(1)), \quad s \searrow \kappa,
\end{gather*}
where $o(1)$ is uniformly small over $x > 1.$ 
\item For $x \in (\eta'(-\infty), \eta'(0))=(0,1)$:
\begin{gather*}
\mathbf{P} \left(\omega(X_s) \leq \mathbb{P}(s) x\right) = \frac{e^{-\mathbb{P}(s) (x \ln x - x + 1)}}{\sqrt{2 \pi \mathbb{P}(s) x}} 
\frac{1}{1-e^{-|\ln x|}} \psi(h) (1 + o(1)), \quad s \searrow \kappa,
\end{gather*}
where $o(1)$ is uniformly small over $x \in (0, 1).$ 
\end{enumerate}
\end{theorem}

\begin{remark} \label{sec3:rem1}
Assume conditions ${\bf{A1-A5}}$ are satisfied and that $X_s$ has the Dirichlet series distribution given by (\ref{distr}). In addition, assume:
\begin{enumerate}
\item[${\bf{A6}}$] there exists real $r > 0$ such that 
$A^{-1}(s) = O((s-\kappa)^r)$, $s \searrow \kappa$.
\end{enumerate}
Then due to proof of Theorem \ref{sec3:prop1} and Lemma \ref{sec3:lem1}, for any $\vartheta > 0,$
$$|\psi_s(z) - \psi(z)| = O(s-\kappa) = o\left(\left(\mathbb{P}(s)\right)^{-\vartheta}\right), \quad s \searrow \kappa.$$  
Thus, $\omega(X_s)$ converges mod-Poisson as $s \searrow \kappa$ at speed $o\left(\left(\mathbb{P}(s)\right)^{-\vartheta}\right)$ for any $\vartheta > 0$. 
That is why we can use Taylor expansions of the functions $\eta(h)$, $\psi(h)$ and in Theorem \ref{sec3:th1} 
instead of the factor $1+o(1)$ to obtain, for any $\vartheta > 0,$ the expansion
\begin{gather*}
1 + \frac{\alpha_1}{\mathbb{P}(s)} + \dotsc + \frac{\alpha_{\vartheta-1}}{\left(\mathbb{P}(s)\right)^{\vartheta-1}} + O\left(\left(\mathbb{P}(s)\right)^{-\vartheta}\right), 
\quad s \searrow \kappa,
\end{gather*}
and for all probabilities these expansions are different.
\end{remark}

Next we establish some Berry-Esseen estimates in the setting of mod-$\phi$ convergence.

\begin{theorem} \label{sec3:th2}
Assume conditions ${\bf{A1-A5}}$ are satisfied and that $X_s$ has the Dirichlet series distribution given by (\ref{distr}).
Let $N_s$ be a Poisson random variable with parameter 
$\mathbb{P}(s)$. Then the inequality 
\begin{gather*}
\sup_{t \in \mathbb{R}} \left|\mathbf{P}\left(\frac{\omega(X_s) - \mathbb{P}(s)}{\sqrt{\mathbb{P}(s)}} \geq t\right) - 
\mathbf{P}\left(\frac{N_s - \mathbb{P}(s)}{\sqrt{\mathbb{P}(s)}} \geq t\right)\right| \leq 
\frac{\pi \mathbb{P}_2(\kappa)}{\sqrt{\mathbb{P}(s)}}
\end{gather*}
holds for all $s>\kappa.$
\end{theorem}

\begin{proof}
Denote $t_s = \mathbb{P}(s)$ as in Theorem \ref{sec3:prop1}.
Noting that 
\begin{gather*}
\mathbf{P}\left(\frac{\omega(X_s) - t_s}{\sqrt{t_s}} \geq t\right) =
\sum_{k=t}^{\infty} \mathbf{P}\left(\frac{\omega(X_s) - t_s}{\sqrt{t_s}} = \frac{[t_s + k \sqrt{t_s}]}{\sqrt{t_s}} - \sqrt{t_s}\right),
\end{gather*}
then using Lemma \ref{sec2:lem1} we obtain
\begin{gather} \label{sec3:eq1}
\mathbf{P}\left(\frac{\omega(X_s) - t_s}{\sqrt{t_s}} \geq t\right) =
\frac{1}{2 \pi} \sum_{k=t}^{\infty} \int_{-\pi}^{\pi} 
e^{-i u \left(\frac{[t_s + k \sqrt{t_s}]}{\sqrt{t_s}} - \sqrt{t_s}\right)} \widetilde \phi_s(iu) du.
\end{gather}

Using the mod-Poisson convergence of $\omega(X_s)$, $s \searrow \kappa$, and the notation $\psi_s(z)$ from the proof of Theorem 
\ref{sec3:prop1} we have 
\begin{gather*}
\widetilde \phi_s(z) \equiv \mathbf{E} e^{z \frac{\omega(X_s) - t_s}{\sqrt{t_s}}} =
e^{- z \sqrt{t_s}} \mathbf{E} e^{\frac{z}{\sqrt{t_s}} \omega(X_s)} =
e^{- z \sqrt{t_s}} e^{t_s \left(e^{z/\sqrt{t_s}} - 1\right)} 
\psi_s \left(z/\sqrt{t_s}\right),
\end{gather*}
where
\begin{gather*}
\psi_s (z) = \exp \left(-t_s (e^z - 1) + \sum_{p \in \mathcal{P}} \ln \left(1 + \alpha_p(s) (e^z - 1)\right)\right).
\end{gather*}
Substituting this in  (\ref{sec3:eq1}) we get
\begin{gather} \label{sec3:eq2}
\mathbf{P}\left(\frac{\omega(X_s) - t_s}{\sqrt{t_s}} \geq t\right) = \\
\nonumber
 \frac{1}{2 \pi} \sum_{k=t}^{\infty} \int_{-\pi}^{\pi} 
e^{-i u \left(\frac{[t_s + k \sqrt{t_s}]}{\sqrt{t_s}} - \sqrt{t_s}\right)} 
e^{- i u \sqrt{t_s}} e^{t_s \left(e^{i u/\sqrt{t_s}} - 1\right)} 
\psi_s \left(i u/\sqrt{t_s}\right) du.
\end{gather}

Let us use the following Taylor expansions as $s \searrow \kappa$
\begin{eqnarray*}
e^{- i u \sqrt{t_s}} e^{t_s \left(e^{i u/\sqrt{t_s}} - 1\right)} =&
e^{- i u \sqrt{t_s}} e^{t_s \left(i u/\sqrt{t_s} - u^2/(2 t_s) + o(1/t_s)\right)}\\
 =& e^{-u^2/2} (1+o(1)),
 \end{eqnarray*} 
 and
 \begin{eqnarray*}
\psi_s \left(i u/\sqrt{t_s}\right) =& \psi_s (0) + 
\frac{i u}{\sqrt{t_s}} \psi'_s (0) - \frac{u^2}{2t_s} \psi''_s (0) + o\left(\frac{u^2}{2 t_s}\right)\\ 
=& 1 + \frac{u^2}{2 t_s} \mathbb{P}_2(s) + o\left(\frac{1}{t_s}\right),
\end{eqnarray*}
where the second equality is true because 
\begin{gather*}
\psi'_s (z) = \psi_s (z) e^z \left(-t_s + \sum_{p \in \mathcal{P}} \frac{A_p(s)}{1 + A_p(s) e^z}\right).
\end{gather*}
Thus, we see that $\psi'_s (0) =0$ and
\begin{gather*}
\psi''_s (z) = \psi_s (z) e^{2 z} \left(\left(-t_s + \sum_{p \in \mathcal{P}} \frac{A_p(s)}{1 + A_p(s) e^z}\right)^2 \right. + \\
+ \left. e^{-z} \left(-t_s + \sum_{p \in \mathcal{P}} \frac{A_p(s)}{1 + A_p(s) e^z}\right) - 
\sum_{p \in \mathcal{P}} \left(\frac{A_p(s)}{1 + A_p(s) e^z}\right)^2\right).
\end{gather*}
Hence, from equation (\ref{sec3:eq2}) we obtain 
\begin{gather*}
\mathbf{P}\left(\frac{\omega(X_s) - t_s}{\sqrt{t_s}} \geq t\right) - 
\mathbf{P}\left(\frac{N_s - t_s}{\sqrt{t_s}} \geq t\right) = \\ =
 \frac{\mathbb{P}_2(s) (1 + o(1))}{4 \pi t_s} \sum_{k=t}^{\infty} 
\int_{-\pi}^{\pi} e^{-i u \left(\frac{[t_s + k \sqrt{t_s}]}{\sqrt{t_s}} - \sqrt{t_s}\right)} u^2 e^{-u^2/2} du\\
= \frac{\mathbb{P}_2(s) (1 + o(1))}{4 \pi t_s} 
\int_{-\pi}^{\pi} \frac{e^{-i u t}}{1 - e^{-iu/\sqrt{t_s}}} u^2 e^{-u^2/2} du\\ 
=
\frac{\mathbb{P}_2(s) (1 + o(1))}{4 \pi \sqrt{t_s}} 
\int_{-\pi}^{\pi} e^{-i u t} e^{-u^2/2} u \,du.
\end{gather*}
Therefore, when $s$ is large enough,
\begin{gather*}
\left|\mathbf{P}\left(\frac{\omega(X_s) - t_s}{\sqrt{t_s}} \geq t\right) - \mathbf{P}\left(\frac{N_s - t_s}{\sqrt{t_s}} \geq t\right)\right| \leq
\frac{\mathbb{P}_2(s)}{2 \pi \sqrt{t_s}} \left|\int_{-\pi}^{\pi} e^{-i u t} e^{-u^2/2} u \, du\right| \leq 
\frac{\pi \mathbb{P}_2(\kappa)}{\sqrt{t_s}}.
\end{gather*}

Similarly we can consider case when 
\begin{gather*}
\mathbf{P}\left(\frac{\omega(X_s) - t_s}{\sqrt{t_s}} \geq t\right) =
\sum_{k=t+1}^{\infty} \mathbf{P}\left(\frac{\omega(X_s) - t_s}{\sqrt{t_s}} = \frac{[t_s + k \sqrt{t_s}]}{\sqrt{t_s}} - \sqrt{t_s}\right)
\end{gather*}
and obtain the analogous result.

Theorem \ref{sec3:th2} is proved.
\end{proof}

We can combine the results of Theorems 
\ref{sec3:th1} and \ref{sec3:th2} and obtain the moderate deviation probabilities for $\omega(X_s)$.

\begin{theorem} \label{sec3:th3}
Assume conditions ${\bf{A1-A5}}$ are satisfied and that $X_s$ has the Dirichlet series distribution given by (\ref{distr}). Let 
$N_s$ be a Poisson random variable with parameter 
$\mathbb{P}(s)$. Then
\begin{gather*}
\mathbf{P}\left(\omega(X_s) \geq \mathbb{P}(s) + \sqrt{\mathbb{P}(s)} y\right) = 
\mathbf{P}\left(N_s \geq \mathbb{P}(s) + \sqrt{\mathbb{P}(s)} y\right) (1 + o(1)), \quad s \searrow \kappa,
\end{gather*}
where $o(1)$ is uniformly small over 
$y = o\left(\sqrt{\mathbb{P}(s)}\right)$, $s \searrow \kappa$.

Let the function $\psi(\cdot)$ be as defined at (\ref{psi}) in Theorem \ref{sec3:prop1}.
If $x = 1 + y / \sqrt{\mathbb{P}(s)}$ 
and $h$ is the solution of $e^h = x$, then
\begin{gather*}
\mathbf{P}\left(\omega(X_s) \geq \mathbb{P}(s) + \sqrt{\mathbb{P}(s)} y\right) = 
\mathbf{P}\left(N_s \geq \mathbb{P}(s) + \sqrt{\mathbb{P}(s)} y\right) \psi(h) (1 + o(1)), \quad s \searrow \kappa, 
\end{gather*}
where $o(1)$ is uniformly small over $y \gg 1$ and 
$y = O\left(\sqrt{\mathbb{P}(s)}\right)$, $s \searrow \kappa$, $x > 1$, and 
\begin{gather*}
\mathbf{P}\left(\omega(X_s) \leq \mathbb{P}(s) + \sqrt{\mathbb{P}(s)} y\right) = 
\mathbf{P}\left(N_s \leq \mathbb{P}(s) + \sqrt{\mathbb{P}(s)} y\right) \psi(h) (1 + o(1)), \quad s \searrow \kappa, 
\end{gather*}
where $o(1)$ is uniformly small over $y \ll 0$ and 
$y = O\left(\sqrt{\mathbb{P}(s)}\right)$, $s \searrow \kappa$, $x \in (0, 1)$.
\end{theorem}

\begin{proof}
We break up the cases depending on $y$.
\begin{enumerate}
\item Let $y = O(1)$. In this case, Theorem \ref{sec3:th2} is satisfied and implies the result. 

\item Let $y = o\left(\sqrt{\mathbb{P}(s)}\right)$, $s \searrow \kappa$, 
and $y \gg 1$. Denote $x = 1 + y/\sqrt{\mathbb{P}(s)}$ and 
let $h$ be the solution of implicit equation $e^h = x$. 
Then according to Cramer's transform of $N_s$ and Theorem \ref{sec3:th1} we obtain the following equalities:
\begin{gather*}
\mathbf{P}\left(\omega(X_s) \geq \mathbb{P}(s) + \sqrt{\mathbb{P}(s)} y\right) = 
\mathbf{P} \left(\omega(X_s) \geq \mathbb{P}(s) x\right) \\
= \frac{e^{-\mathbb{P}(s) (x \ln x - x + 1)}}{\sqrt{2 \pi \mathbb{P}(s) x}} 
\frac{1}{1-x}
\exp \left(-\sum_{n=2}^{\infty} \mathbb{P}_n(\kappa) \frac{(1-x)^n}{n}\right) (1 + o(1)) \\
= \mathbf{P} \left(N_s \geq \mathbb{P}(s) x\right) 
\exp \left(-\sum_{n=2}^{\infty} \mathbb{P}_n(\kappa) \frac{\left(-y/\sqrt{\mathbb{P}(s)}\right)^n}{n}\right)
(1 + o(1)) \\
= \mathbf{P}\left(N_s \geq \mathbb{P}(s) + \sqrt{\mathbb{P}(s)} y\right) (1 + o(1)), \quad s \searrow \kappa.
\end{gather*}
The last equality is true because when $s$ is large enough
\begin{gather*}
\sum_{n=2}^{\infty} \mathbb{P}_n(\kappa) \frac{\left(-y/\sqrt{\mathbb{P}(s)}\right)^n}{n} < \frac{y^2}{\mathbb{P}(s)} \sum_{n=2}^{\infty} \frac{\mathbb{P}_n(\kappa)}{n},
\end{gather*}
where the series $\sum_{n=2}^{\infty} \mathbb{P}_n(\kappa) / n$ converges due to condition ${\bf{A3}}$ there exists $l \geq 1$ such that $b(p_l) > 1$, then we obtain
\begin{gather*}
\sum_{n=2}^{\infty} \frac{\mathbb{P}_n(\kappa)}{n} = \sum_{n=2}^{\infty} \frac{1}{n} 
\left(\sum_{i=1}^l\alpha_{p_i}(\kappa)^n+ \sum_{i>l} \alpha_{p_i}(\kappa)^n\right) \\
\leq \sum_{n=2}^{\infty} \frac{1}{n} 
\left(\sum_{i=1}^l\alpha_{p_i}(\kappa)^n + \frac{C(n, l)}{b(p_l)^n}\right),
\end{gather*}
where according to ${\bf{A4}}$ the series on the right side of the inequality converges.

\item Let $y = o\left(\sqrt{\mathbb{P}(s)}\right)$, $s \searrow \kappa$, 
and $y \ll 0$. Similarly to the previous case, as $s$ tends to $\kappa$ we obtain
\begin{eqnarray*}
\mathbf{P}\left(\omega(X_s) \geq \mathbb{P}(s) + \sqrt{\mathbb{P}(s)} y\right) =& 1 - \mathbf{P}\left(\omega(X_s) < \mathbb{P}(s) + \sqrt{\mathbb{P}(s)} y\right) \\
= &1 - \mathbf{P}\left(N_s < \mathbb{P}(s) + \sqrt{\mathbb{P}(s)} y\right) (1 + o(1))\\
 =&\mathbf{P}\left(N_s \geq \mathbb{P}(s) + \sqrt{\mathbb{P}(s)} y\right) (1 + o(1)).
\end{eqnarray*}

\item Let $y = O\left(\sqrt{\mathbb{P}(s)}\right)$, $s \searrow \kappa$. 
In this case Theorem \ref{sec3:th1} about the large deviation probabilities is satisfied.
\end{enumerate}

Theorem \ref{sec3:th3} is proved.
\end{proof}

\section{Number of prime divisors of $X_s$, counted with multiplicities} \label{sec4}

In this section we consider the random variables $\Omega(X_s)$, where 
$X_s$, $s > \kappa$,
has the Dirichlet series distribution at (\ref{distr}) and conditions ${\bf{A1-A4}}$ are satisfied
together with  the additional condition:
\begin{enumerate}
\item[${\bf{A0}}$] The functions $a(n)$ and $b(n)$ are completely multiplicative.
\end{enumerate}
Note that due to conditions ${\bf{A3}}$ and ${\bf{A0}}$ we have $\sup_{p \in \mathcal{P}} a(p)/b(p)^{\kappa} < 1$ and\\ $\lim_{i \to \infty} a(p_i)/b(p_i)^{\kappa} = 0$.

The following relations hold for any $p \in \mathcal{P}$ and $n \in \mathbb{N}$:
\begin{gather*}
A_{p^{n}}(s) = \sum_{m=n}^{\infty} \left(\frac{a(p)}{b(p)^{s}}\right)^m = \frac{1}{1 - a(p)/(b(p))^s} \left(\frac{a(p)}{b(p)^{s}}\right)^{n}, \\
\mathbf{P}\left(p^{n} | X_s\right) = \frac{A_{p^{n}}(s)}{1+A_p(s)} = \left(\frac{a(p)}{b(p)^{s}}\right)^{n} = \left(\alpha_p(s)\right)^{n},
\end{gather*}
where function $A_{p^{n}}(s)$, $n \in \mathbb{N}$, is defined by (\ref{func_Ap}).

Let $d = \inf_{p \in \mathcal{P}} b(p)^{\kappa}/a(p) > 1$ and assume the following condition is also met:
\begin{enumerate}
\item[${\bf{\widetilde{{\bf{A5}}}}}$] There exists a positive function $\widehat C(n)$ such that $\limsup_{n \to \infty} \widehat C(n)^{1/n} \leq 1$ and such that for any $n \geq 2$ 
\begin{eqnarray*}\label{tildeA5}
\sum_{p \in \mathcal{P}} \left(\frac{a(p)}{b(p)^{\kappa}}\right)^{n} \ln b(p) \leq \frac{\widehat C(n)}{d^n}.
\end{eqnarray*}
\end{enumerate}

First we calculate the mathematical expectation  of $\Omega(X_s)=\sum_p\sum_{n=1}^\infty I_{p^n|X_s}.$
\begin{eqnarray*}
\begin{split}
\mathbf{E} \Omega(X_s) =& \mathbf{E} \sum_{p \in \mathcal{P}} \sum_{n=1}^{\infty} I_{p^{n} | X_s}\\ 
=& \sum_{p \in \mathcal{P}} \sum_{n=1}^{\infty} \left(\frac{a(p)}{b(p)^{s}}\right)^n = \sum_{n=1}^{\infty} \mathbb{P}_{n}(s)\\ 
=&\mathbb{P}(s) + \sum_{n=2}^{\infty} \mathbb{P}_{n}(s), 
\end{split}
\end{eqnarray*}
and the variance of $\Omega(X_s),$
\begin{eqnarray*}
\begin{split}
\mathbf{D} \Omega(X_s) =& \mathbf{E} \Omega^2 (X_s) - \left(\mathbf{E} \Omega(X_s)\right)^2\\ 
=& \mathbb{P}(s) + \sum_{n=2}^{\infty} n \mathbb{P}_{n}(s) -
\sum_{n \neq m}^{\infty} \mathbb{P}_{n+m}(s),
\end{split}
\end{eqnarray*}
where the function $\mathbb{P}_n(s)$, $n \in \mathbb{N}$, is defined by relation (\ref{func_Pn}).

Notice that 
\begin{gather*}
\sum_{n \neq m}^{\infty} \mathbb{P}_{n+m}(s) \leq \left(\sum_{n=2}^{\infty} \mathbb{P}_{n}(s)\right)^2 + \sum_{n=2}^{\infty} \mathbb{P}_{n}(s).
\end{gather*}
Thus, according to ${\bf{A4}},$ as $s \searrow \kappa$
\begin{gather*}
\mathbf{E} \Omega(X_s) = \mathbb{P}(s) + O(1), \quad
\mathbf{D} \Omega(X_s) = \mathbb{P}(s) + O(1).
\end{gather*}

We have the following result about mod-$\phi$ convergence for $\Omega(X_s).$
\begin{theorem} \label{sec4:prop1}
Assume conditions ${\bf{A0-A4}}$ and $\widetilde{{\bf{A5}}}$ are satisfied.
Then mod-Poisson  convergence takes place for the random variables $\Omega(X_s)$, $s > \kappa$, as $s \searrow \kappa$
on $\mathcal{S}_{(-\infty,\,\, \ln d)}$ with parameters 
$t_s = \sum_{n=1}^{\infty} \mathbb{P}_n(s)$ and residue function 
\begin{gather}\label{psi2}
\psi (z) = \exp \left(\sum_{n=2}^{\infty} \mathbb{P}_n(\kappa) \left(\frac{e^{n z} - 1}{n} - (e^z - 1)\right)\right).
\end{gather}
The mod-Poisson convergence is at speed $o(1)$,
$s \searrow \kappa$.
\end{theorem}

\begin{proof}
Using the independence of the events $p^n | X_s$ and $q^m | X_s$ for any irreducible elements $p \neq q$ and taking arbitrary $n, m \in \mathbb{N}$, we get the following: 
\begin{eqnarray*}
\phi_s (z) =& \mathbf{E} e^{z \Omega(X_s)}\\ 
= & \prod_p \mathbf{E} [e^{z\sum_n I_{p^n|X_s}}]\\
=&\prod_p \left(\sum_{m=0}^\infty e^{zm}\mathbf{P}(\sum_{n=1}^\infty I_{p^n|X_s}=m)\right)\\
=&\prod_p \sum_{m=0}^\infty e^{zm}\frac{a(p)}{b(p)^s}(1-\frac{a(p)}{b(p)^s})^{-1}\\
=&\prod_{p \in \mathcal{P}} \left(1 - \frac{a(p)}{b(p)^s}\right)^{-1} 
\frac{1}{(1 - e^z a(p)/(b(p)^s)}.
\end{eqnarray*}

Notice that in the domain $|e^z| < d$ we get
\begin{eqnarray*}
\phi_s (z) =& \exp \left(\sum_{n=1}^{\infty} \mathbb{P}_n(s) \frac{e^{n z} - 1}{n}\right) \\
= &\exp \left((e^z - 1) \sum_{n=1}^{\infty} \mathbb{P}_n(s) + \sum_{n=2}^{\infty} \mathbb{P}_n(s) \left(\frac{e^{n z} - 1}{n} - (e^z - 1)\right)\right)\\ 
= &e^{(e^z - 1) \sum_{n=1}^{\infty} \mathbb{P}_n(s)} \psi_s (z),
\end{eqnarray*}
where $\eta(z) = e^z - 1$ corresponds to the standard Poisson distribution. The series for $\ln \psi_s(z)$ converges for any 
$s \geq \kappa$ since owing to condition ${\bf{A3}}$ there exists $l \geq 1$ such that $b(p_l) > d$ and then due to condition ${\bf{A4}}$ we get
\begin{gather*}
\sum_{n=2}^{\infty} \frac{\mathbb{P}_n(s) \, e^{n z}}{n} = \sum_{n=2}^{\infty} \frac{e^{n z}}{n} 
\left(\left(\frac{a(p_1)}{b(p_1)^s}\right)^n + \dotsb + \left(\frac{a(p_l)}{b(p_l)^s}\right)^n + \sum_{i>l} \left(\frac{a(p_i)}{b(p_i)^s}\right)^n\right) \\
\leq \sum_{n=2}^{\infty} \frac{1}{n} 
\left(\left(\frac{e^z a(p_1)}{b(p_1)^s}\right)^n + \dotsb + \left(\frac{e^z a(p_l)}{b(p_l)^s}\right)^n + C(n, l) \, \left(\frac{e^z a(p_l)}{b(p_l)}\right)^n\right),
\end{gather*}
where according to ${\bf{A4}}$ the series on the right side of the inequality converges.
Therefore, we have obtained that 
$\Omega(X_s)$ converges mod-Poisson  on the strip 
$\mathcal{S}_{(-\infty, \ln d)}$ with parameters 
$t_s \equiv \sum_{n=1}^{\infty} \mathbb{P}_n(s)$ and residue function $\psi(z) \equiv \psi_{\kappa}(z).$ 
This function is well defined on this strip and does not vanish on the real part.  

Let us consider speed of mod-Poisson convergence. 
For any $z \in \mathcal{S}_{(-\infty, \ln d)}$ 
we have
\begin{gather*}
|\psi_s(z) - \psi(z)| = |\psi(z) (f(s, z) - 1)| =
|\psi(z) f_s'(\kappa, z)| (s-\kappa) + o(s-\kappa), \quad s \searrow \kappa,
\end{gather*}
where
\begin{gather*}
f(s, z) = \exp \left(\sum_{n=2}^{\infty} \left(\mathbb{P}_n(s) - \mathbb{P}_n(\kappa)\right) \left(\frac{e^{n z} - 1}{n} - (e^z - 1)\right)\right), \quad f(\kappa, z) = 1.
\end{gather*}
The derivative of $f(s, z)$ by $s$ is equal to
\begin{gather*}
f'_s(s, z) = f(s, z) \sum_{n=2}^{\infty} \left(\sum_{p \in \mathcal P} \frac{- n \, a^n(p) \, \ln b(p)}{b(p)^{n s}}\right) 
\left(\frac{e^{n z} - 1}{n} - (e^z - 1)\right).
\end{gather*}
Hence, if $s = \kappa$ then
\begin{gather*}
f'_s(\kappa, z) = - \sum_{n=2}^{\infty} \left(\sum_{p \in \mathcal P} \frac{a^n(p) \, \ln b(p)}{b(p)^{n \kappa}}\right) (e^{n z} - 1) + 
(e^z - 1) \sum_{n=2}^{\infty} \sum_{p \in \mathcal P} \frac{n \, a^n(p) \, \ln b(p)}{b(p)^{n \kappa}}.
\end{gather*}
Recall that $d > 1.$ We need to show convergence of the series
\begin{gather*}
\sum_{n=2}^{\infty} \left(\sum_{p \in \mathcal P} \frac{a^n(p) \, \ln b(p)}{b(p)^{n \kappa}}\right) e^{n z}
\end{gather*}
for any $z \in \mathcal{S}_{(-\infty, \ln d)}$. 
This is true according to condition $\widetilde{{\bf{A5}}}$.

Therefore, we obtain that $f'_s(\kappa, z)$ is well defined for any 
$z \in \mathcal{S}_{(-\infty, \ln d)}$ and
\begin{gather*}
|\psi_s(z) - \psi(z)| = O(s-\kappa) = o(1), \quad s \searrow \kappa.
\end{gather*}

Theorem \ref{sec4:prop1} is proved.
\end{proof}

\begin{theorem} \label{sec4:th1}
Assume conditions ${\bf{A0-A4}}$ and $\widetilde{{\bf{A5}}}$ are satisfied and that $X_s$ has the Dirichlet series distribution given by (\ref{distr}). 
Let $h$ be defined by the implicit equation $e^h = x$, 
and let $t_s = \sum_{n=1}^{\infty} \mathbb{P}_n(s)$.
\begin{enumerate}
\item The following expansion holds as $s \searrow \kappa,$
\begin{gather*}
\mathbf{P} \left(\Omega(X_s) = t_s x\right) = \frac{e^{-t_s (x \ln x - x + 1)}}{\sqrt{2 \pi t_s x}} e^{ \sum_{n=2}^{\infty} \mathbb{P}_n(\kappa) \left(\frac{x^n - 1}{n} - (x - 1)\right)}
(1 + o(1)),
\end{gather*}
where $o(1)$ is uniformly small over $x \in (0, d)=\left(e^{-\infty},e^{\ln d}\right)$ 
(note that  $h\in(-\infty, \ln d)$),
$t_s x \in \mathbb N$.
\item Similarly, if $d > 1$ then, as $s \searrow \kappa,$
\begin{gather*}
\mathbf{P} \left(\Omega(X_s) \geq t_s x\right) = \frac{e^{-t_s (x \ln x - x + 1)}}{(1-x^{-1})\sqrt{2 \pi t_s x}}  e^{\sum_{n=2}^{\infty} \mathbb{P}_n(\kappa) \left(\frac{x^n - 1}{n} - (x - 1)\right)}
(1 + o(1)),
\end{gather*}
where $o(1)$ is uniformly small over $x \in (1, d) = \left(e^{0},e^{\ln d}\right)$.
\item By applying the result to $-\Omega(X_s)$, one gets as $s \searrow \kappa,$
\begin{gather*}
\mathbf{P} \left(\Omega(X_s) \leq t_s x\right) = \frac{e^{-t_s (x \ln x - x + 1)}}{(1-e^{-|\ln x|})\sqrt{2 \pi t_s x}}  e^{\sum_{n=2}^{\infty} \mathbb{P}_n(\kappa) \left(\frac{x^n - 1}{n} - (x - 1)\right)}
(1 + o(1)), \quad s \searrow \kappa,
\end{gather*}
where $o(1)$ is uniformly small over $x \in (0, 1)= \left(e^{-\infty},e^{0}\right)$.
\end{enumerate}
\end{theorem}

\begin{remark} \label{sec4:rem1}
Assume conditions ${\bf{A0-A4}},\,\widetilde{{\bf{A5}}}$ and ${\bf{A6}}$ are satisfied and that $X_s$ has the Dirichlet series distribution given by (\ref{distr}). 
According to the proof of Theorem \ref{sec4:prop1} and Lemma \ref{sec3:lem1}, the random variables $\Omega(X_s)$, $s > \kappa$, converge mod-Poisson at speed $o\left(t_s^{-\vartheta}\right)$ for any $\vartheta > 0$. 
That is why we can use Taylor expansions of functions $\eta(h)$, $\psi(h)$ and in Theorem \ref{sec4:th1} 
instead of factor $1+o(1)$, $s \searrow \kappa$, obtain for any $\vartheta > 0$ expansion
\begin{gather*}
1 + \frac{\beta_1}{t_s} + \dotsc + \frac{\beta_{\vartheta-1}}{t_s^{\vartheta-1}} + O\left(t_s^{-\vartheta}\right), 
\quad s \searrow \kappa,
\end{gather*}
and for all probabilities these expansions are different.
\end{remark}

Our next goal is to establish some Berry-Esseen estimates in the setting of mod-$\phi$ convergence. 
That is the content of the next Theorem.

\begin{theorem} \label{sec4:th2}
Suppose  conditions ${\bf{A0-A4}}$ and  $\widetilde{{\bf{A5}}}$ are satisfied and let $N_s$ be a Poisson random variable with parameter 
$t_s = \sum_{n=1}^{\infty} \mathbb{P}_n(s)$. Then the inequality 
\begin{gather*}
\sup_{t \in \mathbb{R}} \left|\mathbf{P}\left(\frac{\Omega(X_s) - t_s}{\sqrt{t_s}} \geq t\right) - 
\mathbf{P}\left(\frac{N_s - t_s}{\sqrt{t_s}} \geq t\right)\right| \leq 
\frac{\pi}{\sqrt{t_s}} \left(\sum_{n=1}^{\infty} n \mathbb{P}_{n+1}(\kappa)\right)
\end{gather*}
holds.
\end{theorem}

\begin{proof}
Note that 
\begin{gather*}
\mathbf{P}\left(\frac{\Omega(X_s) - t_s}{\sqrt{t_s}} \geq t\right) =
\sum_{k=t}^{\infty} \mathbf{P}\left(\frac{\Omega(X_s) - t_s}{\sqrt{t_s}} = \frac{[t_s + k \sqrt{t_s}]}{\sqrt{t_s}} - \sqrt{t_s}\right).
\end{gather*}
Then due to Fourier transform and Lemma \ref{sec2:lem1} we obtain
\begin{gather*} \label{sec4:eq1}
\mathbf{P}\left(\frac{\Omega(X_s) - t_s}{\sqrt{t_s}} \geq t\right) =
\frac{1}{2 \pi} \sum_{k=t}^{\infty} \int_{-\pi}^{\pi} 
e^{-i u \left(\frac{[t_s + k \sqrt{t_s}]}{\sqrt{t_s}} - \sqrt{t_s}\right)} \widetilde \phi_s(iu) du,
\end{gather*}
where 
\begin{eqnarray*}\label{tildephi}
\begin{split}
\widetilde \phi_s(z) \equiv& \mathbf{E} e^{z \frac{\Omega(X_s) - t_s}{\sqrt{t_s}}}\\ 
=&e^{- z \sqrt{t_s}} \mathbf{E} e^{\frac{z}{\sqrt{t_s}} \Omega(X_s)}\\
=&e^{- z \sqrt{t_s}} e^{t_s \left(e^{z/\sqrt{t_s}} - 1\right)} 
\psi_s \left(z/\sqrt{t_s}\right)
\end{split}
\end{eqnarray*}
and the function $\psi_s(z)$, $s \geq \kappa$, is defined in the proof of Theorem \ref{sec4:prop1}. 

Similarly to the proof of Theorem \ref{sec3:th2} using mod-Poisson convergence of $\Omega(X_s)$, $s \searrow \kappa$, with parameters $t_s$ and residue function $\psi(z)$ we have 
\begin{gather*}
\mathbf{P}\left(\frac{\Omega(X_s) - t_s}{\sqrt{t_s}} \geq t\right) - 
\mathbf{P}\left(\frac{N_s - t_s}{\sqrt{t_s}} \geq t\right)  \\
\nonumber
= - \frac{1 + o(1)}{4 \pi t_s} \left(\sum_{n=1}^{\infty} n \mathbb{P}_{n+1}(s)\right) \int_{-\pi}^{\pi} \frac{e^{-i u t}}{1 - e^{-iu/\sqrt{t_s}}} u^2 e^{-u^2/2} du  \\
= - \frac{1 + o(1)}{4 \pi \sqrt{t_s}} \left(\sum_{n=1}^{\infty} n \mathbb{P}_{n+1}(s)\right) \int_{-\pi}^{\pi} e^{-i u t} e^{-u^2/2} u \, du
\end{gather*}
because there are the following equalities:
\begin{gather*}
\psi'_s (z) = \psi_s (z) \left(\sum_{n=2}^{\infty} \mathbb{P}_n(s) (e^{n z} - e^z)\right), \\
\psi''_s (z) = \psi_s (z) \left(\sum_{n=2}^{\infty} \mathbb{P}_n(s) (e^{n z} - e^z)\right)^2 + 
\psi_s (z) \left(\sum_{n=2}^{\infty} \mathbb{P}_n(s) (n e^{n z} - e^z)\right).
\end{gather*}
Hence, when $s$ is large enough
\begin{gather*}
\left|\mathbf{P}\left(\frac{\Omega(X_s) - t_s}{\sqrt{t_s}} \geq t\right) - \mathbf{P}\left(\frac{N_s - t_s}{\sqrt{t_s}} \geq t\right)\right| \leq \\ \leq
\frac{1}{2 \pi \sqrt{t_s}} \left(\sum_{n=1}^{\infty} n \mathbb{P}_{n+1}(s)\right) \left|\int_{-\pi}^{\pi} e^{-i u t} e^{-u^2/2} u \, du\right| \leq 
\frac{\pi}{\sqrt{t_s}} \left(\sum_{n=1}^{\infty} n \mathbb{P}_{n+1}(\kappa)\right),
\end{gather*}
where the series on the right hand side of the last inequality converges due to conditions ${\bf{A3}}$ and  ${\bf{A4}}$.

Similarly, we can consider the case when 
\begin{gather*}
\mathbf{P}\left(\frac{\Omega(X_s) - t_s}{\sqrt{t_s}} \geq t\right) =
\sum_{k=t+1}^{\infty} \mathbf{P}\left(\frac{\Omega(X_s) - t_s}{\sqrt{t_s}} = \frac{[t_s + k \sqrt{t_s}]}{\sqrt{t_s}} - \sqrt{t_s}\right)
\end{gather*}
and obtain the analogous result.

Theorem \ref{sec4:th2} is proved.
\end{proof}

We can combine the results of Theorems 
\ref{sec4:th1} and \ref{sec4:th2} and obtain the moderate deviation probabilities for $\Omega(X_s)$.

\begin{theorem} \label{sec4:th3}
Assume conditions ${\bf{A0-A4}}$ and $\widetilde{{\bf{A5}}}$ are satisfied and that $X_s$ has the Dirichlet series distribution given by (\ref{distr}) and let
$N_s$ be a Poisson random variable with parameter 
$t_s = \sum_{n=1}^{\infty} \mathbb{P}_n(s)$. Then as $s \searrow \kappa, $
\begin{gather*}
\mathbf{P}\left(\Omega(X_s) \geq t_s + \sqrt{t_s} y\right) = 
\mathbf{P}\left(N_s \geq t_s + \sqrt{t_s} y\right) (1 + o(1)), 
\end{gather*}
where $o(1)$ is uniformly small over 
$y = o\left(\sqrt{t_s}\right)$, $s \searrow \kappa$.

Further, let the function $\psi(\cdot)$ be as defined at (\ref{psi2}) in Theorem \ref{sec4:prop1}. If 
$x = 1 + y / \sqrt{t_s}$ 
and $h$ is the solution of $e^h = x$ then as $s \searrow \kappa, $
\begin{gather*}
\mathbf{P}\left(\Omega(X_s) \geq t_s + \sqrt{t_s} y\right)
= \mathbf{P}\left(N_s \geq t_s + \sqrt{t_s} y\right) \psi(h) (1 + o(1)),
\end{gather*}
where $o(1)$ is uniformly small over $y \gg 1$ and 
$y = O\left(\sqrt{t_s}\right)$, $s \searrow \kappa$, 
$x \in (1, d)$, and then
\begin{gather*}
\mathbf{P}\left(\Omega(X_s) \leq t_s + \sqrt{t_s} y\right)
= \mathbf{P}\left(N_s \leq t_s + \sqrt{t_s} y\right) \psi(h) (1 + o(1)), \quad s \searrow \kappa, 
\end{gather*}
where $o(1)$ is uniformly small over $y \ll 0$ and 
$y = O\left(\sqrt{t_s}\right)$, $s \searrow \kappa$, $x \in (0, 1)$.
\end{theorem}

\begin{proof}
We consider the cases depending on $y$.
\begin{enumerate}
\item Let $y = O(1)$. In this case Theorem \ref{sec4:th2} is satisfied. 

\item Let $y = o\left(\sqrt{t_s}\right)$, $s \searrow \kappa$. 
Denote $x = 1 + y/\sqrt{t_s}$ and 
let $h$ be the solution of implicit equation $e^h = x$. 
If $y \gg 1$ then according to Cramer's transform of $N_s$ and Theorem \ref{sec4:th1} we obtain the following equalities:
\begin{gather*}
\mathbf{P}\left(\Omega(X_s) \geq t_s + \sqrt{t_s} y\right) = \mathbf{P} \left(\Omega(X_s) \geq t_s x\right) \\
= \frac{e^{-t_s (x \ln x - x + 1)}}{\sqrt{2 \pi t_s x}} 
\frac{1}{1-x}
\exp \left(\sum_{n=2}^{\infty} \mathbb{P}_n(\kappa) \left(\frac{x^n - 1}{n} - (x - 1)\right)\right) (1 + o(1)) \\
= \mathbf{P} \left(N_s \geq t_s x\right) 
\exp \left(\sum_{n=2}^{\infty} \mathbb{P}_n(\kappa) \left(\frac{(1 + y / \sqrt{t_s})^n - 1}{n} - y / \sqrt{t_s}\right)\right)
(1 + o(1)) \\
= \mathbf{P}\left(N_s \geq t_s + \sqrt{t_s} y\right) (1 + o(1)), \quad s\searrow 1.
\end{gather*}

The last equality is true because when $s$ is large enough
\begin{gather*}
\sum_{n=2}^{\infty} \mathbb{P}_n(\kappa) \left(\frac{(1 + y / \sqrt{t_s})^n - 1}{n} - y / \sqrt{t_s}\right) < 
\frac{y^2}{t_s} \sum_{n=1}^{\infty} n \mathbb{P}_{n+1}(\kappa),
\end{gather*}
where the series $\sum_{n=1}^{\infty} n \mathbb{P}_{n+1}(\kappa)$ converges due to conditions ${\bf{A4}}$ and ${\bf{A3}}$.

If $y \ll 0$ then similarly to the previous reasoning and the proof of Theorem \ref{sec3:th3} we obtain the same.

\item Let $y = O\left(\sqrt{t_s}\right)$, $s \searrow \kappa$. 
In this case Theorem \ref{sec4:th1} about the large deviation probabilities is satisfied.
\end{enumerate}

Theorem \ref{sec4:th3} is proved.
\end{proof}

\section{Numeric applications} \label{sec5}

In this section we consider case $R = \mathbb{N}$ with $b(n) = n$, $n \in \mathbb{N}$, and different examples of function $a(n)$. 
Notice that $E = \{1\}$, $\mathcal{P} =: \mathcal{P}_{\mathbb{N}} \subset \mathbb{N}$ is the set of prime numbers and that the
function $b(n)$ is completely multiplicative and greater than $1$.
As in ${\bf{A3}}$, we take an enumeration of $\mathcal{P}_{\mathbb{N}} = \{p_i, i \in \mathbb{N}\}$  such that $p_i < p_{i+1}$ for any $i \geq 1$.

\subsection{The Riemann zeta function}
Let $a(n) = 1$, $n \in \mathbb{N}$, then we get the Riemann zeta function
\begin{gather*}
A(z) = \zeta(z) = \sum_{n=1}^{\infty} \frac{1}{n^z}.
\end{gather*}
This is an analytic function in $\Re z > 1$ and can be expanded to a meromorphic function on $\mathbb{C}$ with a simple pole at the point $z=1,$ so in this case, $\kappa = 1.$ 
It means that in the real case of the argument we get
\begin{gather} \label{func_zeta}
\zeta(s) \sim \frac{1}{s-1}, \quad s\searrow 1.
\end{gather}

In the article \cite{CM1}, the authors considered the Riemann zeta distribution of the random variable $X_s$, $s > 1$, defining it by
\begin{gather*}
\mathbf{P}(X_s = n) = \frac{1}{\zeta(s) n^s}, \quad n \in \mathbb{N}.
\end{gather*}
 Among its many properties are the following:
\begin{gather*}
A_{p^{\alpha}}(s) = \frac{1}{p^{\alpha s} (1 - 1/p^s)}, \quad p \in \mathcal{P}_{\mathbb{N}}, \quad \alpha \in \mathbb{N}, \\
\mathbf{P}(n | X_s) = \frac{1}{n^s}, \quad n \in \mathbb{N},\\
\mathbf{P}(n | X_s, m | X_s) = \mathbf{P}(n | X_s) \mathbf{P}(m | X_s), 
\quad n, m \in \mathbb{N}, \quad \gcd(n, m) = 1.
\end{gather*}

In the papers \cite{CM1}, \cite{CP} the authors researched various properties of $\omega(X_s)$, $\Omega(X_s)$, in particular Central Limit Theorems were obtained for them:
\begin{gather*}
\frac{\omega(X_s) - \mathbb{P}(s)}{\sqrt{\mathbb{P}(s)}} 
\overset{d}{\to} Z, \quad s\searrow 1, \\
\frac{\Omega(X_s) - \mathbb{P}(s)}{\sqrt{\mathbb{P}(s)}} 
\overset{d}{\to} Z, \quad s\searrow 1,
\end{gather*}
where $Z$ is a random variable with standard normal distribution and 
\begin{gather*}
\mathbb{P}(s) = \sum_{p \in \mathcal{P}_{\mathbb{N}}} \frac{1}{p^s}.
\end{gather*}

\begin{theorem}
The random variables $\omega(X_s)$, $s > 1$, converges mod-Poisson as $s\searrow 1$
on $\mathbb{C}$ 
 with parameters 
$$
t_s = \mathbb{P}(s) = \sum_{p \in \mathcal{P}_{\mathbb{N}}} \frac{1}{p^s}
$$ 
and residue function 
\begin{gather*}
\psi(z) = \exp \left(\sum_{p \in \mathcal{P}_{\mathbb{N}}} \left(\ln \left(1 + \frac{e^z - 1}{p}\right) - \frac{e^z - 1}{p}\right)\right)
\end{gather*} 
at speed $o\left(t_s^{-\vartheta}\right)$ for any $\vartheta > 0$.
Moreover, Theorems \ref{sec3:th1}, \ref{sec3:th2} and \ref{sec3:th3} are true for the Riemann zeta distribution.
\end{theorem}
\begin{proof}
We show the conditions ${\bf{A1-A6}}$ are satisfied. The first three conditions ${\bf{A1-A3}}$ are obvious. 
Condition ${\bf{{\bf{A4}}}}$ is true because 
\begin{gather*}
\mathbb{P}_n(s) = \mathbb{P}(n s) = \frac{1}{2^{n s}} + \dotsc + \frac{1}{\widetilde{p}^{n s}} + O\left(\frac{1}{\widetilde{p}^{n s}}\right),
\quad n \geq 2, \quad s \geq 1.
\end{gather*}
Condition ${\bf{A6}}$ is true because of relation (\ref{func_zeta}).
Let us check ${\bf{A5}}$. The derivative of $A_p(s)$ is 
\begin{gather*}
A'_p(s) = - \frac{p^s \ln p}{(p^s - 1)^2}.
\end{gather*}
This leads to the series:
\begin{gather*}
\sum_{p \in \mathcal{P}_{\mathbb{N}}} A'_p(s) A_p(s) = - \sum_{p \in \mathcal{P}_{\mathbb{N}}} \frac{p^s \ln p}{(p^s - 1)^3}
\end{gather*}
that converges for any $s \geq 1$.

Hence, we can apply Theorem \ref{sec3:prop1} and remark \ref{sec3:rem1} to conclude the proof.
\end{proof}
Notice that in the Berry-Esseen estimate (Theorem \ref{sec3:th2}) for the constant $\mathbb{P}_2(\kappa)$ we have the following inequality:
\begin{gather*}
\mathbb{P}(2) <  \sum_{n=1}^{\infty} \frac{1}{n^2} =  \frac{\pi^2}{6}.
\end{gather*}

\begin{theorem}
The random variables $\Omega(X_s)$, $s > 1$, converge mod-Poisson as $s\searrow 1$
on $\mathcal{S}_{(-\infty, \ln 2)}$ 
with parameters 
$$
t_s = \sum_{n=1}^{\infty} \mathbb{P}(n s) = \sum_{n=1}^{\infty} \sum_{p \in \mathcal{P}_{\mathbb{N}}} \frac{1}{p^{n s}}
$$ 
and residue function 
\begin{gather*}
\psi (z) = \exp \left(\sum_{n=2}^{\infty} \mathbb{P}(n) \left(\frac{e^{n z} - 1}{n} - (e^z - 1)\right)\right)
\end{gather*} 
at speed $o\left(t_s^{-\vartheta}\right)$ for any $\vartheta > 0$.
Moreover Theorems \ref{sec4:th1}, \ref{sec4:th2} and \ref{sec4:th3} are true for the Riemann zeta distribution.
\end{theorem} 
\begin{proof}
The function $a(n)$ is completely multiplicative.
Therefore, condition ${\bf{A0}}$ is true.
Condition $\widetilde{{\bf{A5}}}$ is fulfilled because $d = \inf_{p \in \mathcal{P}_{\mathbb{N}}} p = 2$ and for any $n \geq 2$ we have
\begin{eqnarray*}
2^n \sum_{p \in \mathcal{P}_{\mathbb{N}}} \frac{\ln p}{p^n} =& \sum_{p \in \mathcal{P}_{\mathbb{N}}} \frac{\ln (p/2) + \ln 2}{(p/2)^n} \\
\leq& \ln 2 + \sum_{k=1}^{\infty} \frac{\ln k + \ln 2}{k^n} \\
\leq &\left(1 + \frac{\pi^2}{6}\right) \ln 2 + \sum_{k=1}^{\infty} \frac{\ln k}{k^2} \\
\equiv& \widehat C(n),
\end{eqnarray*}
where the series $\sum_{k=1}^{\infty} \ln k/k^2$ converges.
Hence, we can apply Theorem \ref{sec4:prop1} and remark \ref{sec4:rem1}. 
\end{proof}

Notice that due to the power series  $\sum_{n=1}^{\infty} n \mathbb{P}_{n+1} (\kappa)$ appearing as the constant in the Berry-Esseen estimate (Theorem \ref{sec4:th2}) we obtain the following estimate.
\begin{eqnarray*}
\sum_{n=1}^{\infty} n \mathbb{P}(n+1) <&\sum_{n=1}^{\infty} n \left(\frac{1}{2^{n+1}} + \frac{1}{3^{n+1}} + \frac{1}{3^{n+1}} \sum_{k=1}^{\infty} \frac{1}{k^{n+1}}\right) \\
<& \sum_{n=1}^{\infty} \frac{n}{2^{n+1}} + \left(1 + \frac{\pi^2}{6}\right) \sum_{n=1}^{\infty} \frac{n}{3^{n+1}} \\
= &\frac{(1/2)^2}{(1 - 1/2)^2} + \left(1 + \frac{\pi^2}{6}\right) \frac{(1/3)^2}{(1 - 1/3)^2}\\
 =& 1 + \frac{1}{4} \left(1 + \frac{\pi^2}{6}\right) = \frac{30 + \pi^2}{24} \\
 < &1.67.
\end{eqnarray*}

\begin{remark}
Let $a(n) = n^{\gamma}$, $n \in \mathbb{N}$, $\gamma \in \mathbb{R}$, then we get $A(s) = \zeta(s-\gamma)$, $s > \gamma + 1$, and can use the previous example about the Riemann zeta distribution.
\end{remark}

\subsection{The Dedekind function}
Let 
\begin{gather*}
a(n) = n \prod_{p \in \mathcal{P}_{\mathbb{N}}, \, p | n} \left(1 + \frac{1}{p}\right), \quad n \in \mathbb{N},
\end{gather*}
be the Dedekind function. 
Then we get the Dedekind Dirichlet series
\begin{gather} \label{func_ded}
A(z) = \sum_{n=1}^{\infty} \frac{a(n)}{n^z} = \frac{\zeta(z) \zeta(z-1)}{\zeta(2z)}
\end{gather}
that is an analytic function in $\Re z > 2,$ so $\kappa = 2.$ 

Define the Dedekind Dirichlet series distribution of the random variable 
$X_s$, $s > 2$, by
\begin{eqnarray}\label{Ded}
\mathbf{P}(X_s = n) = \frac{a(n)}{A(s) n^s}, \quad n \in \mathbb{N}.
\end{eqnarray}

This has the following associated quantities:
\begin{gather*}
A_p(s) = \frac{p + 1}{p^s - p}, \quad p \in \mathcal{P}_{\mathbb{N}},\\
\mathbf{P}(p | X_s) = \alpha_p(s) = \frac{p+1}{p^s+1}, \quad p \in \mathcal{P}_{\mathbb{N}}.
\end{gather*}

\begin{theorem}
The random variables $\omega(X_s)$, $s > 2$, where $X_s$ has the distribution in (\ref{Ded}) converge mod-Poisson as $s \searrow 2$
on $\mathbb{C}$
with parameters 
$$
t_s = \sum_{p \in \mathcal{P}_{\mathbb{N}}} \frac{p+1}{p^s+1}
$$
and residue function 
\begin{gather*}
\psi(z) = \exp \left(\sum_{p \in \mathcal{P}_{\mathbb{N}}} \left(\ln \left(1 + \frac{p+1}{p^2+1} (e^z - 1)\right) - \frac{p+1}{p^2+1} (e^z - 1)\right)\right)
\end{gather*} 
at speed $o\left(t_s^{-\vartheta}\right)$ for any $\vartheta > 0$.
Moreover, Theorems \ref{sec3:th1}, \ref{sec3:th2} and \ref{sec3:th3} are true for the Dedekind zeta distribution.
\end{theorem}
\begin{proof}
We check  that conditions ${\bf{A1-A6}}$ are satisfied for $s>2.$ The first four conditions ${\bf{A1-A4}}$ are obvious (look at the first example for ${\bf{A4}}$). 
Condition ${\bf{A6}}$ is true because of relation (\ref{func_zeta}) and equation (\ref{func_ded}).
Let us check ${\bf{A5}}$. The derivative of the function $A_p(s)$ is given by
\begin{gather*}
A'_p(s) = - \frac{p^s (p+1) \ln p}{(p^s - p)^2}.
\end{gather*}
Thus, we get the following series:
\begin{gather*}
\sum_{p \in \mathcal{P}_{\mathbb{N}}} A'_p(s) A_p(s) = - \sum_{p \in \mathcal{P}_{\mathbb{N}}} \frac{p^s (p+1)^2 \ln p}{(p^s - p)^3}
\end{gather*}
that converges for any $s \geq 2$.

Hence, we can apply Theorem \ref{sec3:prop1} and remark \ref{sec3:rem1} to conclude the proof.
\end{proof}
Similarly to the first example with the Riemann zeta distribution, in the Berry-Esseen estimate (Theorem \ref{sec3:th2}) for the constant $\mathbb{P}_2(\kappa)$ we have the following bound:
\begin{eqnarray*}
\mathbb{P}_2(2) =& \sum_{p \in \mathcal{P}_{\mathbb{N}}} \left(\frac{p+1}{p^2+1}\right)^2\\
 <& \sum_{p \in \mathcal{P}_{\mathbb{N}}} \frac{(p+1)^2}{p^4} \\
 = &\sum_{p \in \mathcal{P}_{\mathbb{N}}} \frac{1}{p^2} + \sum_{p \in \mathcal{P}_{\mathbb{N}}} \frac{2}{p^3} + \sum_{p \in \mathcal{P}_{\mathbb{N}}} \frac{1}{p^4}  \\
< &\frac{1}{2^2} + \frac{1}{3^2} + \frac{1}{3^2} \zeta(2) + 2 \left(\frac{1}{2^3} + \frac{1}{3^3} + \frac{1}{3^3} \zeta(3)\right) + \frac{1}{2^4} + \frac{1}{3^4} + \frac{1}{3^4} \zeta(4) \\
 <& 0.55 + 0.42 + 0.09\\
  =& 1.06.
\end{eqnarray*}

\subsection{The Euler totient function}
Let 
\begin{gather*}
a(n) = n \prod_{p \in \mathcal{P}_{\mathbb{N}}, \, p | n} \left(1 - \frac{1}{p}\right), \quad n \in \mathbb{N},
\end{gather*}
be the Euler totient function. 
Then we get the Euler zeta function
\begin{gather*}
A(z) = \sum_{n=1}^{\infty} \frac{a(n)}{n^z} = \frac{\zeta(z-1)}{\zeta(z)}
\end{gather*}
that is an analytic function in $\Re z > 2,$ so $\kappa = 2.$ 

Let us define the Euler zeta distribution of the random variable 
$X_s$, $s > 2$, by
\begin{eqnarray}\label{Eul}
\mathbf{P}(X_s = n) = \frac{a(n)}{A(s) n^s}, \quad n \in \mathbb{N},
\end{eqnarray}
with the following associated quantities:
\begin{gather*}
A_p(s) = \frac{p - 1}{p^s - p}, \quad p \in \mathcal{P}_{\mathbb{N}},\\
\mathbf{P}(p | X_s) = \alpha_p(s) = \frac{p-1}{p^s-1}, \quad p \in \mathcal{P}_{\mathbb{N}}.
\end{gather*}

Similar to the  previous example of the Dedekind Dirichlet series distribution, conditions ${\bf{A1-A6}}$ are satisfied
and we can apply Theorem \ref{sec3:prop1} and remark \ref{sec3:rem1}. 
\begin{theorem} The random variables $\omega(X_s)$, $s > 2$,where $X_s$ has the distribution at (\ref{Eul}) converge mod-Poisson as $s \searrow 2$
on $\mathbb{C}$
with parameters 
$$
t_s = \sum_{p \in \mathcal{P}_{\mathbb{N}}} \frac{p-1}{p^s-1}
$$
and residue function 
\begin{gather*}
\psi(z) = \exp \left(\sum_{p \in \mathcal{P}_{\mathbb{N}}} \left(\ln \left(1 + \frac{p-1}{p^2-1} (e^z - 1)\right) - \frac{p-1}{p^2-1} (e^z - 1)\right)\right)
\end{gather*} 
at speed $o\left(t_s^{-\vartheta}\right)$ for any $\vartheta > 0$.
Moreover,  Theorems \ref{sec3:th1}, \ref{sec3:th2} and \ref{sec3:th3} are true for the Euler zeta distribution.
\end{theorem}
Similarly to the first example with the Riemann zeta distribution, in the Berry-Esseen estimate (Theorem \ref{sec3:th2}) for constant $\mathbb{P}_2(\kappa)$ we have the following inequality:
\begin{gather*}
\mathbb{P}_2(2) = \sum_{p \in \mathcal{P}_{\mathbb{N}}} \left(\frac{p-1}{p^2-1}\right)^2 
= \sum_{p \in \mathcal{P}_{\mathbb{N}}} \frac{1}{(p+1)^2} 
< \sum_{p \in \mathcal{P}_{\mathbb{N}}} \frac{1}{p^2} < 0.55.
\end{gather*}

\subsection{Number of positive divisors function}
Let 
\begin{gather*}
a(n) = \sum_{d | n} 1, \quad n \in \mathbb{N},
\end{gather*}
be the number of positive divisors function. 
Then we get the following Dirichlet series:
\begin{gather} \label{func_num}
A(z) = \sum_{n=1}^{\infty} \frac{a(n)}{n^z} = \zeta^2(z)
\end{gather}
that is an analytic function in $\Re z > 1,$ so $\kappa = 1.$ 

Define the distribution of a random variable 
$X_s$, $s > 1$, by
\begin{eqnarray}\label{posdiv}
\mathbf{P}(X_s = n) = \frac{a(n)}{A(s) n^s}, \quad n \in \mathbb{N}.
\end{eqnarray}

This distribution has the following associated quantities:
\begin{gather*}
A_p(s) = \frac{2 p^s - 1}{(p^s - 1)^2}, \quad p \in \mathcal{P}_{\mathbb{N}},\\
\mathbf{P}(p | X_s) = \alpha_p(s) = \frac{1}{p^s} \left(2 - \frac{1}{p^s}\right), \quad p \in \mathcal{P}_{\mathbb{N}}.
\end{gather*}

\begin{theorem} 
The random variables $\omega(X_s)$, $s > 1$, where $X_s$ has the distribution at (\ref{posdiv}) converge mod-Poisson as $s\searrow 1$
on $\mathbb{C}$
with parameters 
$$
t_s = \sum_{p \in \mathcal{P}_{\mathbb{N}}} \frac{2 p^s - 1}{p^{2 s}}
$$ 
and residue function 
\begin{gather*}
\psi(z) = \exp \left(\sum_{p \in \mathcal{P}_{\mathbb{N}}} \left(\ln \left(1 + \frac{2 p - 1}{p^{2 }} (e^z - 1)\right) - \frac{2 p - 1}{p^{2 }} (e^z - 1)\right)\right)
\end{gather*} 
at speed $o\left(t_s^{-\vartheta}\right)$ for any $\vartheta > 0$.
Moreover, Theorems \ref{sec3:th1}, \ref{sec3:th2} and \ref{sec3:th3} are true for this case.
\end{theorem}
\begin{proof}
We show that conditions ${\bf{A1-A6}}$ are satisfied. The first four conditions ${\bf{A1-A4}}$ are obvious (look at the first example for ${\bf{A4}}$). 
Condition ${\bf{A6}}$ is true because of the relation (\ref{func_zeta}) and equation (\ref{func_num}).
Let us check ${\bf{A5}}$. The derivative of the function $A_p(s)$ is given by
\begin{gather*}
A'_p(s) = s \ln p \left(\frac{2}{p^s - 1} + \frac{2 (2 p^s - 1)}{(p^s - 1)^3}\right).
\end{gather*}
This leads to the following series:
\begin{gather*}
\sum_{p \in \mathcal{P}_{\mathbb{N}}} A'_p(s) A_p(s) = s \sum_{p \in \mathcal{P}_{\mathbb{N}}} \frac{\ln p \, (2 p^s - 1)}{(p^s - 1)^2} \left(\frac{2}{p^s - 1} + \frac{2 (2 p^s - 1)}{(p^s - 1)^3}\right)
\end{gather*}
that converges for any $s \geq 1$.

Hence, we can apply Theorem \ref{sec3:prop1} and remark \ref{sec3:rem1} to finish the proof.
\end{proof}
Similarly to the first example with the Riemann zeta distribution, in the Berry-Esseen estimate (Theorem \ref{sec3:th2}) for the constant $\mathbb{P}_2(\kappa)$ we have the following bound:
\begin{gather*}
\mathbb{P}_2(1) = \sum_{p \in \mathcal{P}_{\mathbb{N}}} \left(\frac{2 p - 1}{p^2}\right)^2 < \sum_{p \in \mathcal{P}_{\mathbb{N}}} \frac{4}{p^2} < 2.2.
\end{gather*}

\begin{remark}
Similarly to $R = \mathbb{N}$, we can consider $R = \mathbb{Z}$ with $b(n) = |n|$, $n \in \mathbb{Z}$, where $E = \{1, -1\}$, $\mathcal{P} =: \mathcal{P}_{\mathbb{Z}} = \{p, -p : p \in \mathcal{P}_{\mathbb{N}}\} \subset \mathbb{Z}$.
According to condition ${\bf{A3}}$, we can take an enumeration of $\mathcal{P}_{\mathbb{Z}} = \{p_i, i \in \mathbb{N}\}$  such that $|p_i| \leq |p_{i+1}|$ for any $i \geq 1$. Then
due to the fact that $a(n) = a(-n)$, $n \in \mathbb{Z} \setminus \{0\}$, we get
\begin{gather*}
A(z) = \sum_{n \in \mathbb{Z} \setminus \{0\}} \frac{a(n)}{|n|^z} = 2 \sum_{n=1}^{\infty} \frac{a(n)}{n^z}.
\end{gather*}
This is an analytic function in $\Re z > \kappa$ for some real $\kappa$.
And we can consider the distribution of the integer-valued non-zero random variable $X_s$, $s > \kappa$, that is defined by (\ref{distr}).
\end{remark}

\section{Non-numeric applications} \label{sec6}

Now we consider examples of non-numeric sets $R$ with some functions $a(n)$ and $b(n)$. 

\subsection{Random polynomials}

In this subsection we demonstrate an application of mod-Poisson convergence tools in the context of randomly selected monic polynomials over a finite field $\mathbb{F}_q$, where $q$ is prime or a power of a prime.
We denote by $\mathbb{F}_q[x]$ the polynomials over $\mathbb{F}_q$ and define a norm on $\mathbb{F}_q[x]$ by
\begin{gather*}
|F| = q^{\deg F}, \quad F \in \mathbb{F}_q[x].
\end{gather*}
Notice that $|F G| = |F| \cdot |G|$ for any non-zero $F, G \in F_q[x]$.

Denote by $\mathbb{P}_q[x]$ and $\mathbb{M}_q[x]$ the monic irreducible polynomials and monic polynomials, respectively, in $\mathbb{F}_q[x]$. 
Denote by $\mathbb{P}_{q, k}[x]$ and $\mathbb{M}_{q, k}[x]$ the monic irreducible polynomials and monic polynomials of degree $k$, respectively.
We note that  as in \cite{R},
\begin{gather*}
\# \mathbb{M}_{q, k}[x] = q^k, \\
\# \mathbb{P}_{q, k}[x] = \frac{q^k}{k} + O\left(\frac{q^{k/2}}{k}\right), \quad k \to \infty.
\end{gather*}
The analog of the Riemann zeta function $\zeta$ is then
\begin{gather*}
\zeta_q(z) = \sum_{F \in \mathbb{M}_{q}[x]} \frac{1}{|F|^z} = 
\prod_{P \in \mathbb{P}_{q}[x]} \left(1-\frac{1}{|P|^z}\right)^{-1}.
\end{gather*}

Let us consider $R = \mathbb{M}_q[x]$, where $E = \{1\}$, $\mathcal{P} = \mathbb{P}_q[x]$.
Let $a(F) = 1$ and $b(F) = |F|$ for any $F \in \mathbb{M}_q[x].$ These functions are completely multiplicative and $b(F) \geq 1$.
Hence, $A(z) = \zeta_q(z)$.
From the first expression for $\zeta_q$ we derive for $s > 1$ the following Dirichlet series:
\begin{gather} \label{zeta_q}
\zeta_q(s) = \sum_{k=0}^{\infty} \sum_{F \in \mathbb{M}_{q, k}[x]} \frac{1}{|F|^s} = \sum_{k=0}^{\infty} \sum_{F \in \mathbb{M}_{q, k}[x]} q^{-k s} = \sum_{k=0}^{\infty} \frac{q^k}{q^{k s}} = \frac{1}{1-q^{1-s}}.
\end{gather}
We remark that this makes it clear that $\zeta_q(z)$ is analytic function in $\Re z > 1$ (so $\kappa = 1$) with no zeroes on $\Re z = 1$. It is meromorphic with a simple pole at $z=1$ with a residue of $1/\ln q$.

We now describe how to sample a monic polynomial from $\mathbb{F}_q[x]$ using $\zeta_q(z)$ and examine some simple properties of the sampled random polynomial. 
Fix $s>1$ and denote the randomly selected polynomial by $F_s$. It will have distribution given by
\begin{gather*}
\mathbf{P}(F_s = F) = \frac{1}{\zeta_q(s) |F|^s}, \quad F \in \mathbb{M}_q[x].
\end{gather*}
We will call this the $q$-zeta distribution. The random polynomial $F_s$ has a factorization into a product of irreducible monic polynomials, and writing $F | F_s$ to denote that the monic polynomial $F$ appears in the factorization of $F_s$.
In this section, we will establish large deviation and Berry-Esseen estimates for  $\omega(F_s),$ the number of distinct irreducible factors and $\Omega(F_s),$ the total number of irreducible factors of $F_s.$ The $q$-zeta distribution has the following associated quantities:
\begin{gather*}
A_{P^{\alpha}}(s) = \frac{1}{|P|^{\alpha s} (1 - 1/|P|^s)}, \quad P \in \mathbb{P}_q[x], \quad \alpha \in \mathbb{N},\\
\mathbf{P}(F | F_s) = \frac{1}{|F|^s}, \quad F \in \mathbb{M}_q[x], \\
\mathbf{P}(F | F_s, G | F_s) = \frac{1}{|F|^s} \frac{1}{|G|^s} = \mathbf{P}(F | F_s) \mathbf{P}(G | F_s), \quad F, G \in \mathbb{M}_q[x].
\end{gather*}
\begin{theorem} 
The random variables $\omega(F_s)$, $s > 1$, converge mod-Poisson as $s\searrow 1$ on $\mathbb{C}$ with parameters 
$$
t_s = \mathbb{P}(s) = \sum_{P \in \mathbb{P}_q[x]} \frac{1}{|P|^s}
$$ 
and residue function 
\begin{gather*}
\psi(z) = \exp \left(\sum_{P \in \mathbb{P}_q[x]} \left(\ln \left(1 + \frac{e^z - 1}{|P|}\right) - \frac{e^z - 1}{|P|}\right)\right)
\end{gather*} 
at speed $o\left(\left(\mathbb{P}(s)\right)^{-\vartheta}\right)$ for any $\vartheta > 0$.
In addition, Theorems \ref{sec3:th1}, \ref{sec3:th2} and \ref{sec3:th3} hold for $\omega(F_s)$ under the $q$-zeta distribution.
\end{theorem}
\begin{proof}
We show the conditions ${\bf{A1-A6}}$  are satisfied in this case. The first three conditions ${\bf{A1-A3}}$ are obvious, where for condition ${\bf{A3}}$ we choose an enumeration of $\mathbb{P}_q[x] = \{P_i, i \in \mathbb{N}\}$  that satisfies $|P_i| \leq |P_{i+1}|$ for any $i \geq 1$. 
Let us check condition ${\bf{A6}}$. Due to equation (\ref{zeta_q}) we get  $\zeta_q(z)$ has a simple pole at $z=1$ with residue  $1/\ln q.$
This implies that condition ${\bf{A6}}$ is fulfilled with $r = 1$.

We now consider condition ${\bf{A4}}.$ Notice that in the present case $\mathbb{P}_n(s) = \mathbb{P}(n s).$ In addition,
\begin{eqnarray} \label{eq_polyn_A4}
\begin{split}
\sum_{P : |P| > q^l} \frac{1}{|P|^{n s}} =& \sum_{k=l+1}^{\infty} \sum_{P \in \mathbb{P}_{q, k}[x]} \frac{1}{q^{k n s}} \\
=& \sum_{k=l+1}^{\infty} \frac{1}{q^{k n s}} \left(\frac{q^k}{k} + O\left(\frac{q^{k/2}}{k}\right)\right)\\
 \leq& \delta \sum_{k=l+1}^{\infty} \frac{1}{q^{k n s}} \frac{q^k}{k} \\
\leq &\frac{\delta}{l+1} \sum_{k=l+1}^{\infty} q^{k(1-n s)}\\
=& \frac{\delta}{l+1} \frac{q^{(l+1)(1-n s)}}{1 - q^{1-n s}}\\ 
\leq& \frac{2 \delta q^{l+1}}{(l+1) q^{n s}} \frac{1}{q^{l n s}}\\
 \leq& \frac{2 \delta q^{l+1}}{(l+1) q^{n}} \frac{1}{q^{l n s}}.
\end{split}
\end{eqnarray}
This means that in (\ref{A4})  we can take $C(n, l) = 2 \delta q^{l+1} / ((l+1) q^n)$ for some positive number $\delta$. 

Now we check condition ${\bf{A5}}$. The derivative of function $A_P(s)$ is equal to
\begin{gather*}
A'_P(s) = - \frac{|P|^s \ln |P|}{(|P|^s - 1)^2}.
\end{gather*}
This leads to the following series:
\begin{eqnarray*}
\sum_{P \in \mathbb{P}_q[x]} A'_P(s) A_P(s) =& - \sum_{P \in \mathbb{P}_q[x]} \frac{|P|^s \ln |P|}{(|P|^s - 1)^3}\\
 = & \ln q \, \sum_{k=1}^{\infty} \sum_{P \in \mathbb{P}_{q, k}[x]} \frac{k \,q^{k s}}{(q^{k s} - 1)^3} \\ 
=& \ln q \, \sum_{k=1}^{\infty} \frac{k \,q^{k s}}{(q^{k s} - 1)^3} \left(\frac{q^k}{k} + O\left(\frac{q^{k/2}}{k}\right)\right),
\end{eqnarray*}
that converges for any $s \geq 1$.

Hence, we can apply Theorem \ref{sec3:prop1} and remark \ref{sec3:rem1} to finish the proof.
\end{proof}

Similar to relation (\ref{eq_polyn_A4}), in the Berry-Esseen estimate (Theorem \ref{sec3:th2}) for the constant $\mathbb{P}_2(\kappa)$ we have the following inequality:
\begin{gather*}
\mathbb{P}(2) = \sum_{P \in \mathbb{P}_q[x]} \frac{1}{|P|^2} < \frac{\delta}{q - 1}.
\end{gather*}

\begin{theorem}
The random variables $\Omega(F_s)$, $s > 1$, converge mod-Poisson as $s\searrow 1$
on $\mathcal{S}_{(-\infty, \ln q)}$ 
with parameters 
$$
t_s = \sum_{n=1}^{\infty} \mathbb{P}(n s) = \sum_{n=1}^{\infty} \sum_{P \in \mathbb{P}_q[x]} \frac{1}{|P|^{n s}}
$$ 
and residue function 
\begin{gather*}
\psi (z) = \exp \left(\sum_{n=2}^{\infty} \mathbb{P}(n) \left(\frac{e^{n z} - 1}{n} - (e^z - 1)\right)\right)
\end{gather*} 
at speed $o\left(t_s^{-\vartheta}\right)$ for any $\vartheta > 0$.
In addition,  Theorems \ref{sec4:th1}, \ref{sec4:th2} and \ref{sec4:th3} are true for  $\Omega(F_s)$ under the $q$-zeta distribution.
\end{theorem}
\begin{proof}
We only need check that ${\bf{A0}}$ and $\widetilde{{\bf{A5}}}$ are satisfied. Recall that the functions $a(F)$ and $b(F)$ are completely multiplicative.
Therefore, condition ${\bf{A0}}$ is true.
Condition $\widetilde{{\bf{A5}}}$ is fulfilled because $d = \inf_{P \in \mathbb{P}_q[x]} |P| = q$ and for any $n \geq 2$ we have
\begin{eqnarray*}
q^n \sum_{P \in \mathbb{P}_q[x]} \frac{\ln |P|}{|P|^n} = &
q^n \ln q \, \sum_{k=1}^{\infty} \sum_{P \in \mathbb{P}_{q, k}[x]} \frac{k}{q^{n k}}\\
 =& q^n \ln q \, \sum_{k=1}^{\infty} \frac{k}{q^{n k}} \left(\frac{q^k}{k} + O\left(\frac{q^{k/2}}{k}\right)\right)  \\
\leq &\delta \ln q \, \sum_{k=1}^{\infty} q^{k(1 - n - n/k)}\\
 \leq &  \delta \ln q \, \left(q + 1 + \sum_{k=3}^{\infty} q^{k(1 - 2 n/3)}\right)\\
 \leq & \delta \ln q \, \left(q + 1 + \sum_{k=3}^{\infty} q^{-k/3}\right)\\
  \equiv &\widehat C(n),
\end{eqnarray*}
where the series $\sum_{k=3}^{\infty} q^{-k/3}$ converges.

Hence, we can apply Theorem \ref{sec4:prop1} and remark \ref{sec4:rem1} to finish the proof. 
\end{proof}

We remark that similarly to relation (\ref{eq_polyn_A4}) for constant $\sum_{n=1}^{\infty} n \mathbb{P}_{n+1} (\kappa)$ in the Berry-Esseen estimate (Theorem \ref{sec4:th2}) we obtain the following relations
\begin{gather*}
\sum_{n=1}^{\infty} n \mathbb{P}(n+1) \leq
\delta \sum_{n=1}^{\infty} n q^{-n} = \frac{\delta q}{(q-1)^2}.
\end{gather*}

\subsection{Dedekind Domains}

In this subsection we demonstrate how mod-Poisson convergence can be applied to sampling random ideals in Dedekind domains. 
We start with a number field $K$, i.e. an extension field of $\mathbb{Q}$ of finite degree. 
The ring of integers of $K$, denoted by $\mathcal{O}_K$, is the ring of all algebraic integers in $K$. 
An algebraic integer is the root of a monic polynomial with integer coefficients. 
In any $\mathcal{O}_K$, there will be unique (up to order) factorization of ideals into a product of prime ideals.

There is a well-defined norm on ideals in $\mathcal{O}_K$.
This is denoted by $N$ and is defined as
\begin{gather*}
N(I) = [\mathcal{O}_K : I] = |\mathcal{O}_K/I|.
\end{gather*}
A key property of the norm is that it is multiplicative: for any ideals $I$, $J$ in $\mathcal{O}_K$,
\begin{gather*}
N(I \, J) = N(I) N(J).
\end{gather*}
Define $\mathcal{I}_K = \{I :I \text{ is an ideal in } \mathcal{O}_K\}$.
Let $\mathcal{P}_K \subset \mathcal{I}_K$ be the set of prime ideals in $\mathcal{O}_K$.
Further, we will use 
Landau's prime number Theorem (PNT).
\begin{theorem} (Landau) \label{Landau}
If for $x > 0$, $\rho(x) = \#\{\mathfrak{p} \in \mathcal{P}_K | N(\mathfrak{p}) \leq x\}$ then
\begin{gather*}
\rho(x) \sim \frac{x}{\ln x}, \quad x \to \infty.
\end{gather*}
\end{theorem}

Let us consider $R = \mathcal{I}_K$, where $E = \{\mathcal{O}_K\}$, $\mathcal{P} = \mathcal{P}_K$ is the set of prime ideals in $\mathcal{O}_K$.
Let $a(I) = 1$ and $b(I) = N(I)$ for any $I \in \mathcal{I}_K \setminus \{0\}$, these functions are completely multiplicative and $b(I) \geq 1$.
Hence, we get the following series
\begin{gather*}
A(z) = \zeta_K(z) = \sum_{I \in \mathcal{I}_K \setminus \{0\}} \frac{1}{N(I)^z} = 
\prod_{\mathfrak{p} \in \mathcal{P}_K} \left(1-\frac{1}{N(\mathfrak{p})^s}\right)^{-1}.
\end{gather*}
This is called the Dedekind zeta function and is analytic in $\Re z > 1$ (so $\kappa = 1$).
As in the case of the Riemann zeta function, the Dedekind zeta function has an extension to a meromorphic function on $\mathbb{C}.$
This extension has a simple pole at $z = 1$ with residue $c_K,$ which is  the class number (see \cite{M} or \cite{L}).
The actual value of $c_K$ will not play a role in our development.

We can now define an $\mathcal{I}_K$-valued random variable,
$I_s$ for $s > 1,$ that has distribution under some fixed probability measure $\mathbf{P}$ given by
\begin{gather*}
\mathbf{P}(I_s = I) = \frac{1}{\zeta_K(s) N(I)^s}, \quad I \in \mathcal{I}_K \setminus \{0\}.
\end{gather*}
We call this the Dedekind zeta distribution. 
The random ideal $I_s$ has a factorization into a product of prime ideals, and writing $I | I_s$ to denote that the ideal $I$ appears in the factorization of $I_s$. We will establish large deviation and Berry-Esseen estimates for  $\omega(I_s),$ the number of distinct factors, and $\Omega(I_s),$ the total number of factors factors of $I_s.$
This distribution has the following associated quantities:
\begin{gather*}
A_{\mathfrak{p}^{\alpha}}(s) = \frac{1}{N(\mathfrak{p})^{\alpha s} (1 - 1/N(\mathfrak{p})^s)}, \quad \mathfrak{p} \in \mathcal{P}_K, \quad \alpha \in \mathbb{N},\\
\mathbf{P}(I | I_s) = \frac{1}{N(I)^s}, \quad I \in \mathcal{I}_K \setminus \{0\}, \\
\mathbf{P}(\mathfrak{p}^\alpha | I_s, \mathfrak{q}^\beta | I_s) = \frac{1}{N(\mathfrak{p})^{\alpha s}} \frac{1}{N(\mathfrak{q})^{\beta s}} = \mathbf{P}(\mathfrak{p}^\alpha | I_s) \mathbf{P}(\mathfrak{q}^\beta | I_s), \quad \mathfrak{p}\not=\mathfrak{q} \in \mathcal{P}_K.
\end{gather*}
\begin{theorem}
The random variables $\omega(I_s), \,\,s > 1$, converge mod-Poisson as $s\searrow 1$ on $\mathbb{C}$ with parameters 
$$
t_s = \mathbb{P}(s) = \sum_{\mathfrak{p} \in \mathcal{P}_K} \frac{1}{N(\mathfrak{p})^s}
$$ 
and residue function 
\begin{gather*}
\psi(z) = \exp \left(\sum_{\mathfrak{p} \in \mathcal{P}_K} \left(\ln \left(1 + \frac{e^z - 1}{N(\mathfrak{p})}\right) - \frac{e^z - 1}{N(\mathfrak{p})}\right)\right)
\end{gather*} 
at speed $o\left(\left(\mathbb{P}(s)\right)^{-\vartheta}\right)$ for any $\vartheta > 0$.

Moreover, Theorems \ref{sec3:th1}, \ref{sec3:th2} and \ref{sec3:th3} hold for $\omega(I_s)$ under the Dedekind zeta distribution.
\end{theorem}
\begin{proof}
We show the conditions ${\bf{A1-A6}}$ are satisfied.  The first three conditions ${\bf{A1-A3}}$ are obvious, where for condition ${\bf{A3}}$ we choose an enumeration of $\mathcal{P}_K = \{\mathfrak{p}_i, i \in \mathbb{N}\}$  such that $N(\mathfrak{p}_i) \leq N(\mathfrak{p}_{i+1})$ for all $i \geq 1$. 
Condition ${\bf{A6}}$ is fulfilled with $r = 1$ because according to \cite{M} (chapter 7) there exists the limit
\begin{gather*}
\lim_{s\searrow 1} \frac{\zeta_K(s)}{\zeta(s)}=c_K,
\end{gather*}
that is positive and because of relation (\ref{func_zeta}) for the Riemann zeta function.

Let us check condition ${\bf{A4}}$.
Notice that $\mathbb{P}_n(s) = \mathbb{P}(n s)$.
Put $f(t) = t^{-n s}$. We note that by Landau's PNT \ref{Landau}, $\lim_{k \to \infty} f(k) \rho(k) = 0$. Also, $\rho(1) = 0$. 
Then using Abel summation by parts, since the "boundary terms" $f \cdot \rho(\infty) = f(1) \rho(1) = 0$, we get the following series for $l \geq 2$:
\begin{eqnarray}
\begin{split}
\label{eq_ideal_A4}
\sum_{\mathfrak{p} : N(\mathfrak{p}) > l} \frac{1}{N(\mathfrak{p})^{n s}} = &\sum_{k=l+1}^{\infty} \sum_{\mathfrak{p} : N(\mathfrak{p}) = k} \frac{1}{k^{n s}}\\
=& \sum_{k=l+1}^{\infty} (\rho(k) - \rho(k-1)) f(k) \\ 
=& \int_{l}^{\infty} \rho_{[t]} f'(t) \, dt\\
\sim& n s \int_l^{\infty} \frac{t}{\ln t} t^{-n s-1} \, dt \\
\leq & \frac{n s}{\ln l} \int_l^{\infty} t^{- n s} \, dt  \\ 
=& \frac{l}{(1 - 1/(n s)) \ln l} \frac{1}{l^{n s}}\\
\leq& \frac{l}{(1 - 1/n) \ln l} \frac{1}{l^{n s}},
\end{split}
\end{eqnarray}
where we consider $l \geq 2$, otherwise if $l = 1$ the integral will be from $2$. 
It means that we can take $C(n, l) = \delta n l / ((n - 1) \ln l)$ for some positive number $\delta$.

Now we check condition ${\bf{A5}}$. The derivative of the function $A_{\mathfrak{p}}(s)$ is equal to
\begin{gather*}
A'_{\mathfrak{p}}(s) = - \frac{N(\mathfrak{p})^s \ln N(\mathfrak{p})}{(N(\mathfrak{p})^s - 1)^2}.
\end{gather*}
Similarly, due to Landau's PNT and Abel summation by parts, we get the following series:
\begin{eqnarray*}
\sum_{\mathfrak{p} \in \mathcal{P}_K} A'_{\mathfrak{p}}(s) A_{\mathfrak{p}}(s) = &
- \sum_{\mathfrak{p} \in \mathcal{P}_K} \frac{N(\mathfrak{p})^s \ln N(\mathfrak{p})}{(N(\mathfrak{p})^s - 1)^3} \\
= &\sum_{k=2}^{\infty} \sum_{\mathfrak{p} : N(\mathfrak{p}) = k} \frac{k^s \ln k}{(k^s - 1)^3}\\
=& \sum_{k=2}^{\infty} (\rho(k) - \rho(k-1)) \frac{k^s \ln k}{(k^s - 1)^3}\\
 =& \int_{2}^{\infty} \rho_{[t]} \left(\frac{t^s \ln t}{(t^s - 1)^3}\right)'_t \, dt  \\
\sim& \int_2^{\infty} \frac{t}{\ln t} \left(\frac{t^s \ln t}{(t^s - 1)^3}\right)'_t \, dt \\
=& \int_2^{\infty} \frac{t^s}{(t^s - 1)^3 \ln t} \, dt + \int_2^{\infty} \frac{s \, t^s}{(t^s - 1)^3} \, dt - \int_2^{\infty} \frac{3 s \, t^{2 s}}{(t^s - 1)^4} \, dt,
\end{eqnarray*}
where all three integrals are finite for any $s \geq 1$.

Hence, we can apply Theorem \ref{sec3:prop1} and remark \ref{sec3:rem1}. 
\end{proof}
Similarly to relation (\ref{eq_ideal_A4}), in the Berry-Esseen estimate (Theorem \ref{sec3:th2}) for constant $\mathbb{P}(2 \kappa)$ we have the following inequality:
\begin{gather*}
\mathbb{P}(2) = \sum_{\mathfrak{p} \in \mathcal{P}_K} \frac{1}{N(\mathfrak{p})^2} < \frac{\delta}{2 \ln 2}.
\end{gather*}

We can also obtain the large dviation results for $\Omega(I_s),$  the total number of ideal factors of $I_s$ counting multiplicities.
\begin{theorem}
The random variables $\Omega(I_s)$, $s > 1$, converge mod-Poisson as $s\searrow 1$
on $\mathcal{S}_{(-\infty, \ln 2)}$ 
with parameters 
$$
t_s = \sum_{n=1}^{\infty} \mathbb{P}(n s) = \sum_{n=1}^{\infty} \sum_{\mathfrak{p} \in \mathcal{P}_K} \frac{1}{N(\mathfrak{p})^{n s}}
$$ 
and residue function 
\begin{gather*}
\psi (z) = \exp \left(\sum_{n=2}^{\infty} \mathbb{P}(n) \left(\frac{e^{n z} - 1}{n} - (e^z - 1)\right)\right)
\end{gather*} 
at speed $o\left(t_s^{-\vartheta}\right)$ for any $\vartheta > 0$.
In addition,  Theorems \ref{sec4:th1}, \ref{sec4:th2} and \ref{sec4:th3} are true for $\Omega(I_s)$ under the Dedekind zeta distribution.
\end{theorem}
\begin{proof}Recall that functions $a(I)$ and $b(I)$ are completely multiplicative.
Therefore, condition ${\bf{A0}}$ is true. 
Let us check condition $\widetilde{{\bf{A5}}}$. 
Due to Landau's PNT we have that $d = \inf_{\mathfrak{p} \in \mathcal{P}_K} N(\mathfrak{p}) \geq 2$ is finite and that similarly to (\ref{eq_ideal_A4}) for any $n \geq 2$
\begin{eqnarray*}
\sum_{\mathfrak{p} \in \mathcal{P}_K} \frac{\ln N(\mathfrak{p})}{N(\mathfrak{p})^n} \sim& - \int_d^{\infty} \frac{t}{\ln t} \left(\frac{\ln t}{t^n}\right)'_t \, dt \\
=& \int_d^{\infty} \frac{n - 1/\ln t}{t^n}  \, dt \\
\leq& (n - \ln d) \int_d^{\infty} t^{-n} \, d t\\
 =& \frac{(n - \ln d) d}{n-1} d^{-n}.
\end{eqnarray*}
It means that we can take $\widehat C(n) = \widetilde \delta (n - \ln d) n/(n-1)$ for some positive number $\widetilde \delta$.

Hence, we can apply Theorem \ref{sec4:prop1} and remark \ref{sec4:rem1}. 
\end{proof}
We remark that similarly to (\ref{eq_ideal_A4}) for constant $\sum_{n=1}^{\infty} n \mathbb{P}_{n+1} (\kappa)$ in the Berry-Esseen estimate (Theorem \ref{sec4:th2}) we obtain the following relations
\begin{gather*}
\sum_{n=1}^{\infty} n \mathbb{P}(n+1) \leq \sum_{n=1}^{\infty} n \frac{2 \delta (n+1)}{n \ln 2} 2^{-(n+1)} = \frac{2 \delta}{\ln 2} \sum_{n=2}^{\infty} n \, 2^{-n} = \frac{2 \delta}{\ln 2} (4 - 1/2) = \frac{3 \delta}{\ln 2}.
\end{gather*}

\end{document}